\documentclass[11pt,a4paper,reqno]{amsart}
\usepackage[all]{xy}
\SelectTips{cm}{}
\calclayout

\usepackage{amscd,mathtools,amssymb,amsopn,amsmath,amsthm,graphics,mathrsfs,accents}
\usepackage[colorlinks=true,linkcolor=red,citecolor=blue]{hyperref}
\usepackage[mathscr]{euscript}
\usepackage[usenames]{color}
\usepackage[all]{xy}

\newcommand{\rt}{\rightarrow}

\newcommand{\lrt}{\longrightarrow}
\newcommand{\lf}{\leftarrow}
\newcommand{\llf}{\longleftarrow}

\newcommand{\st}{\stackrel}

\newcommand{\im}{\mathsf{im}}
\def\ker{\operatorname{\mathsf{ker}}}
\def\cok{\operatorname{\mathsf{coker}}}

\newcommand{\D}{\mathbb{D} }
\newcommand{\Q}{\mathcal{Q} }
\newcommand{\Z}{\mathbb{Z} }
\newcommand{\I}{\mathcal{I} }

\newcommand{\C}{\mathscr{C} }

\newcommand{\G}{\mathcal{G} }

\newcommand{\m}{\mathfrak{m}}

\newcommand{\J}{\mathcal{J} }

\newcommand{\F}{\mathcal{F} }

\newcommand{\co}{{\scriptstyle \bf C}\frak{oh}}
\newcommand{\p}{\mathfrak{p} }

\newcommand{\op}{{\rm{op}}}

\def\proj{\operatorname{\mathsf{proj}}}
\def\inj{\operatorname{\mathsf{inj}}}

\newcommand{\id} {\mathsf{id}}

\def\fl{\operatorname{\mathsf{fl}}}

\def\Ext{\operatorname{{\mathsf{Ext}}}}

\def\hom{\operatorname{{\mathsf{Hom}}}}

\def\Tor{\operatorname{\mathsf{Tor}}}
\def\hom{\operatorname{\mathsf{Hom}}}
\def\H{\operatorname{\mathsf{H}}}

\def\O{\mathcal{O}}
\def\F{\mathcal{F}}

\def\G{\mathcal{G}}

\def\md{\operatorname{\mathsf{mod}}}

\def\si{\mathsf{\Sigma}}
\def\syz{\mathsf{\Omega}}
\def\coh{\operatorname{\mathsf{coh}}}
\def\qcoh{\operatorname{\mathsf{Qcoh}}}

\def\g{\mathbf g}
\def\a{\mathsf {a}}
\def\x{\mathsf {X}}

\def\pp{\mathsf {P}}

\def\f{\mathsf {F}}

\def\fl{\mathsf {Flat}}
\def\ch{\mathsf {Ch}}
\def\h{\mathbf h}
\def\L{\mathcal L}
\def\p{\mathcal P}
\def\A{\mathcal A}

\def\f{\mathbf f}
\def\g{\mathbf g}

\def\h{\mathbf h}
\def\U{\mathcal U}
\def\co{\operatorname{\mathsf{cone}}}

\def\ruf{\operatorname{\mathsf{RUF}}}
\def\luf{\operatorname{\mathsf{LUF}}}

\def\al{\operatorname{\boldsymbol{\alpha}}}
\def\be{\operatorname{\boldsymbol{\beta}}}
\def\ga{\operatorname{\boldsymbol{\gamma}}}
\def\th{\operatorname{\boldsymbol{\theta}}}
\def\ep{\operatorname{\boldsymbol{\epsilon}}}

\def\et{\operatorname{\boldsymbol{\eta}}}

\newtheorem{theorem}{Theorem}[section]

\newtheorem{cor}[theorem]{Corollary}

\newtheorem{lem}[theorem]{Lemma}
\newtheorem{prop}[theorem]{Proposition}

\theoremstyle{definition}
\newtheorem{dfn}[theorem]{Definition}
\newtheorem{example}[theorem]{Example}

\newtheorem{conv}[theorem]{Convention}

\newtheorem{rem}[theorem]{Remark}
\newtheorem{s}[theorem]{}

\theoremstyle{plain}

\theoremstyle{definition}

\numberwithin{equation}{section}
\usepackage{lineno}

\begin{document}

\title[phantom stable categories of $n$-Frobenius categories]
{phantom stable categories of $n$-Frobenius categories}
\author[Bahlekeh, Fotouhi, Salarian and Sartipzadeh]{Abdolnaser Bahlekeh, Fahimeh Sadat Fotouhi, Shokrollah Salarian and Atousa Sartipzadeh}

\address{Department of Mathematics, Gonbad Kavous University, Postal Code:4971799151, Gonbad Kavous, Iran}
\email{bahlekeh@gonbad.ac.ir}

\address{ School of Mathematics, Institute for Research in Fundamental Science (IPM), P.O.Box: 19395-5746, Tehran, Iran}\email{ffotouhi@ipm.ir}

\address{Department of Pure Mathematics, Faculty of Mathematics and Statistics,
University of Isfahan, P.O.Box: 81746-73441, Isfahan,
 Iran and \\ School of Mathematics, Institute for Research in Fundamental Science (IPM), P.O.Box: 19395-5746, Tehran, Iran}
 \email{Salarian@ipm.ir}

\address{Department of Pure Mathematics, Faculty of Mathematics and Statistics,
University of Isfahan, P.O.Box: 81746-73441, Isfahan, Iran}
 \email{asartipz@gmail.com}

\subjclass[2010]{18E10, 18G15, 18E35, 14F05}

\keywords{$n$-Frobenius category;  phantom stable category; $n$-$\Ext$-phantom morphism; semi-separated noetherian scheme}
\thanks{{This work is based upon research funded by Iran National Science Foundation (INSF) under project No. 4001480.} The research of the second author was in part supported by a grant from IPM}

\begin{abstract}
Let $n$ be a non-negative integer. An exact category $\C$ is said to be an $n$-Frobenius category, provided that
it has enough $n$-projectives and $n$-injectives and the $n$-projectives coincide with the $n$-injectives. It is proved that any abelian category with non-zero $n$-projective objects, admits a non-trivial $n$-Frobenius subcategory. In particular, we explore several examples of $n$-Frobenius categories. Also, as a far-reaching generalization of the stabilization of a Frobenius category, we introduce and study the notion of the phantom stable category of an $n$-Frobenius category $\C$. Precisely,
assume that $\p\subseteq\Ext^n_{\C}$ is the subfunctor consisting of all conflations of length $n$ factoring through $n$-projective objects. A couple $(\C_{\p}, T)$, where $\C_{\p}$ is an additive category and $T$ is a covariant additive functor from $\C$ to $\C_{\p}$, is a phantom stable category of $\C$, provided that for any morphism $f$ in $\C$, $T(f)=0$, whenever $f$ is an $n$-$\Ext$-phantom morphism and $T(f)$ is an isomorphism in $\C_{\p}$, if $f$ acts as invertible on $\Ext^n/{\p}$, and $T$ has the universal property with respect to these conditions. The main focus of this paper is to show that the phantom stable category of an $n$-Frobenius category always exists. Some properties of phantom stable categories that reveal the efficiency of these categories are studied.
\end{abstract}
\maketitle

\tableofcontents

\section{Introduction}

Assume that $\A$ is an abelian category and $\C$ is a full additive subcategory of $\A$ which is closed under extensions. It is known that, the exact structure of $\A$ is inherited by $\C$, see \cite[Lemma 10.20]{buh}. {Assume that $n$ is a non-negative integer.} An extension which is obtained by splicing of $n$ conflations in $\C$, will be called a {\em conflation of length $n$}. For arbitrary objects $A, B\in\C$, the equivalence classes of all conflations of length $n$ in $\C$, $\Ext^n_{\C}(A, B)$, forms an abelian group with respect to the Baer sum operation. {In the case $n=0$, we set $\Ext^0_{\C}(A, B):=\hom_{\C}(A, B)$.} Moreover, the notion of $n$-projective and $n$-injective objects in $\C$, are defined in terms of the vanishing of $\Ext^{n+1}$ functor. Now we call $\C$ an {\it $n$-Frobenius category}, provided that it has enough $n$-projectives and $n$-injectives, and the $n$-projectives coincide with the $n$-injectives. Assume that $\C$ is an $n$-Frobenius category. Then, for any $k\geq 1$, a given object $N\in\C$ fits into conflations of length $k$, say $\syz^kN\rt P_{k-1}\rt\cdots\rt P_0\rt N$, where $P_i$'s are $n$-projective, which will be called {\it unit conflations}. Also, $\syz^kN$ is said to be a $k$-th syzygy of $N$. We set $\mathsf{H}:=\bigcup_{M, N\in\C}\hom_{\C}(M, N)$ and $\Ext^n:=\bigcup_{M, N\in\C}\Ext^n_{\C}(M, \syz^nN)$, where $\syz^nN$ runs over all the $n$-th syzygies of $N$. For any $f,g\in\H$, the pull-back and push-out of any conflation of length $n$ along $f$ and $g$ are again conflations of length $n$. These operations, which abbreviately are denoted by $\Ext^nf$ and $g\Ext^n$, respectively, induce an $\H$-bimodule structure on $\Ext^n$. We denote by $\p$ the subfunctor of $\Ext^n_{\C}$ consisting of all conflations of the form $\Ext^nf$, for some $f:M\rt P$ of $\H$, where $P$ is an $n$-projective object of $\C$. Particularly, Proposition \ref{equal} indicates that $\p$ is a submodule of $\Ext^n$, and then, $\Ext^n/{\p}$ will be an $\H$-bimodule. It is proved in Section 6 that if $f$ annihilates $\Ext^n/{\p}$ from the left, then it annihilates this module from the right and vice versa. Similarly, we establish that an element $g\in\H$ acts as invertible on $\Ext^n/{\p}$ from the left and the right, simultaneously, see Corollaries \ref{lr} and \ref{qo}. Now consider two classes of morphisms in $\H$, as follows:
\begin{itemize}\item The class of all morphisms in $\H$ annihilating $\Ext^n/{\p}$, that is called $n$-$\Ext$-phantom morphisms. For the historical remark on phantom morphisms, see \ref{s100}. \item The class of all morphisms in $\H$ acting as invertible on $\Ext^n/{\p}$, which will be called quasi-invertible morphisms.
\end{itemize}

By a {\it phantom stable category of $\C$}, we mean an additive category $\C_{\p}$, together with a covariant additive functor $T:\C\lrt\C_{\p}$ such that:\\ (1) $T(s)$ is an isomorphism in $\C_{\p}$, for any quasi-invertible morphism $s$. \\(2) For any $n$-$\Ext$-phantom morphism $\varphi$, $T(\varphi)=0$ in $\C_{\p}$. \\(3) Any covariant additive functor $T':\C\lrt\D$ satisfying the conditions (1) and (2), factors in a unique way through $T$.

In this paper, first we provide some important examples of $n$-Frobenius categories, and then we show that for any $n$-Frobenius category $\C$, the phantom stable category $(\C_{\p}, T)$ exists, see Theorem \ref{thmst}. Our formalism reveals that the phantom stabilization of an $n$-Frobenius category is an efficient and natural extension of the classical stable category of a Frobenius category. Indeed, in the case $n=0$, morphisms factoring through projective objects are $n$-$\Ext$-phantom morphisms, which is an ideal of $\H$. Particularly, in order to examine the efficiency of phantom stable categories, it is proved that for given two objects $M, N\in\C$ and arbitrary syzygies $\syz M$ and $\syz N$ (with respect to $n$-projectives), there is an induced map $\hom_{\C_{\p}}(M, N)\lrt\hom_{\C_{\p}}(\syz M, \syz N)$, which is an isomorphism, see Theorem \ref{syziso}.


Our motivation in studying the $n$-Frobenius categories and then introducing the concept of the phantom stable categories comes from the fact that there are categories that rarely have enough projectives or even, have no projective objects. However, they have often enough $n$-projectives or their subcategories of $n$-projective objects are non-trivial, for some integer $n\geq 1$. For instance, there are no projective objects in the category of quasi-coherent sheaves over the projective line $\mathbf{P^1}(R)$, where $R$ is a commutative ring with identity, see \cite[Corollary 2.3]{ee} and also \cite[Exercise III 6.2(b)]{har}. However, as proved by Serre, the category of coherent sheaves over a projective scheme, has enough locally free sheaves, see \cite[Corollary 5.18]{har}. More generally, the argument given in the proof of \cite[Lemma 1.12]{or} reveals that for a semi-separated noetherian scheme $\x$ of finite Krull dimension, there exists a non-negative integer $n$ such that locally free sheaves of finite rank, are $n$-projective objects in the category of coherent sheaves, $\coh(\x)$, and in particular, the subcategory of $n$-projective objects of $\coh(X)$, is non-trivial.

Assume that $\A$ is an abelian category such that its subcategory of $n$-projective objects is non-trivial, for some integer $n\geq 1$. Then it is proved that $\A$ admits an $n$-Frobenius subcategory $\C$, see Theorem \ref{subcat}. We should stress that this fact is conceivable, as it is known that any abelian category with non-zero projective objects admits a 0-Frobenius subcategory. Assume that $\x$ is a semi-separated noetherian scheme of finite Krull dimension. As we have already mentioned, the category of locally free sheaves of finite rank, $\L$, is a subcategory of $n$-$\proj\coh(\x)$, for some non-negative integer $n$. We will see that the subcategory consisting of all syzygies of complete resolutions of locally frees of finite rank, $\C(\L)$, is an $n$-Frobenius subcategory of $\coh\x$, and in particular, $n$-$\proj\C(\L)=\L$, see Proposition \ref{locally}.

In order to explore more examples of phantom stable categories, we consider the category of complexes of flat $\O_{\x}$-modules, $\ch(\fl\x)$, over the scheme $\x$. Pursuing the argument given in \cite[page 28]{ha1}, yields that this is an exact category, with the exact structure being short exact sequences of complexes. We will see in Theorem \ref{fp}, that the aforementioned category is $n$-Frobenius, for some integer $n\geq 0$, and in particular, its $n$-projective objects (and then $n$-injective objects), are exactly the flat complexes, i.e., those acyclic complexes with flat kernels, see \cite[Definition 2.5]{e}.

The paper is organized as follows. In Section 2, we study $n$-Frobenius categories and explore some examples of such categories. {It will be observed that semi-separated noetherian schemes of finite Krull dimension are good venues for searching such examples.} It is shown that any abelian category with non-zero $n$-projective objects admits an $n$-Frobenius subcategory. Assume that $\x$ is a semi-separated noetherian scheme of finite Krull dimension. Then we will see that $\C(\L)$ is an $n$-Frobenius subcategory of $\coh(\x)$, for some integer $n$. Also, it is proved that the category $\C(\fl\x)$ consisting of all syzygies of complete resolution of flat sheaves, is an $n$-Frobenius subcategory of $\qcoh(\x)$, for some non-negative integer $n$, and $n$-$\proj\C(\fl\x)=\fl\x$. In Section 3, we study those morphisms in $\H$ acting as invertible on $\Ext^{n+1}_{\C}$. It is shown that a given morphism $f$ acts as invertible on $\Ext^{n+1}_{\C}$ from the left if and only if it acts as invertible from the right. These morphisms will be called quasi-invertible morphisms.
Section 4 is devoted to studying conflations factoring through $n$-projective objects, namely, those arising from pull-back along morphisms ending at $n$-projectives. This class of conflations forms a subfunctor of $\Ext^n_{\C}$ and will be denoted by $\p$. It is proved that a given conflation lies in $\p$ if and only if it is obtained from push-out along a morphism starting at an $n$-projective object, and so, $\p$ is an $\H$-submodule of $\Ext^n$. In Section 5, we show that any conflation $\ga$ in $\C$ can be represented as a pull-back as well as a push-out of unit conflations. These representations will be called a right (left) unit factorization of $\ga$. In Section 6, we investigate those morphisms annihilating the module $\Ext^{n+1}_{\C}$. It is shown that a given object $f\in\H$ annihilates $\Ext^{n+1}_{\C}$ from the left if and only if it annihilates from the right. We call such a morphism $f$, as an $n$-$\Ext$-phantom morphism. In Section 7, we introduce a composition operator $``\circ"$ on $\Ext^n/{\p}$. It is shown that this operator is associative, also distributive over the Baer sum on both sides. In particular, for any object $M$, $\Ext^n_{\C}(M, \syz^nM)/{\p}$ has a ring structure with identity, and also, there exists a ring homomorphism $\hom_{\C}(M, M)\lrt\Ext^n_{\C}(M, \syz^nM)/{\p}$ sending quasi-invertible morphisms to invertible elements. In Section 8, we consider an equivalence relation on $\Ext^n/{\p}$. It is observed that this relation is compatible with the composition $``\circ"$ as well as the Baer sum operation. In the paper's final section, we show that the phantom stable category $(\C_{\p}, T)$ of an $n$-Frobenius category $\C$, always exists.

\section{$n$-Frobenius categories}

Let $\A$ be an abelian category and let $n$ be a non-negative integer. In this section, we study subcategories of $\A$ having enough $n$-projectives and $n$-injectives, and the class of $n$-projectives coincides with the class of $n$-injectives, which we call $n$-Frobenius categories. It is shown that if $\A$ has non-zero $n$-projective objects, then it admits a non-trivial $n$-Frobenius subcategory. Also, several examples of $n$-Frobenius categories are presented. Let us begin this section by stating our convention.

\begin{conv}Throughout the paper, $\A$ is an abelian category and $\C$ is a full additive subcategory of $\A$ which is closed under extensions. So, as we mentioned in the introduction, $\C$ becomes an exact category. We also assume that $(\x, \O_{\x})$ is a semi-separated, noetherian scheme of finite Krull dimension and all locally free sheaves are assumed to be of finite rank.
\end{conv}

\begin{dfn} We say that an extension of length $t\geq 1$, $0\rt B\rt X_{t-1}\rt\cdots\rt X_0\rt A\rt 0$ in $\A$, is a {\it conflation} of length $t$ in $\C$, provided that it is obtained by splicing $t$ conflations of length 1 in $\C$ and it will be denoted by $B\rt X_{t-1}\rt\cdots\rt X_{0}\rt A$. The set of all equivalence classes of such conflations of length $t$ will be depicted by $\Ext^t_{\C}(A, B)$. Also set $\Ext^0_{\C}(A, B):=\hom_{\C}(A, B)$. It is easily seen that for any $i\geq 0$, $\Ext^i_{\C}(-, -):\C^{\op}\times\C\lrt\mathbf{Ab}$ is a bifunctor.
Recall that two conflations $\ep, \ep'\in\Ext^t_{\C}(A, B)$ are equivalent, provided that there is a chain of conflations of length $t$, $\ep=\ep_0, \ep_1\cdots, \ep_k=\ep'$ such that for any $0\leq i\leq k-1$, we have either a morphism $\ep_i\rt\ep_{i+1}$ or a morphism $\ep_{i+1}\rt\ep_i$ with fixed ends, see \cite[Chapter VII, Proposition 3.1]{mit}.
Following Keller \cite{ke,ke1} conflations of length 1, will be called just conflations. In particular, if $B\st{f}\rt C\st{g}\rt A$ is a conflation, then $f$ (resp. $g$) is called an inflation (resp. a deflation). So if $\ep:B\rt X_{t-1}\rt\cdots\rt X_0\rt A$ is a conflation of length $t$, then $\ep=\ep_{t-1}\cdots\ep_0$, where $\ep_i^,s$ are conflations with compatible ends.
\end{dfn}

\begin{dfn}Let $n$ be a non-negative integer.\\
(1) A given object $P\in\C$ (resp. $I\in\C$) is said to be an {\it $n$-projective} (resp. {\it $n$-injective}) object of $\C$, if $\Ext^i_{\C}(P, X)=0$ (resp. $\Ext^i_{\C}(X, I)=0$) for all integers $i>n$ and all objects $X\in\C$. The class of all $n$-projective (resp. $n$-injective) objects will be denoted by $n$-$\proj\C$ (resp. $n$-$\inj\C$). \\
(2) The category $\C$ is said to have enough $n$-projectives, provided that each object $M$ in $\C$ fits into a deflation $P\rt M$ with $P$ $n$-projective. Dually one has the notion of having enough $n$-injectives.\\
(3) We say that the exact category $\C$ is {\it $n$-Frobenius}, if
$\C$ has enough $n$-projectives and $n$-injectives, and $n$-$\proj\C$ coincides with $n$-$\inj\C$.
\end{dfn}

\begin{rem}\label{remexam}(1) It follows from the definition that $0$-Frobenius categories are indeed the usual notion of Frobenius categories.\\ (2) If $\C$ is an $n$-Frobenius category, then it will be an $i$-Frobenius category, for any $i\ge n$. In particular, $n$-$\proj\C= i$-$\proj\C$.
\end{rem}

{
\begin{rem}

In the literature, there have appeared two concepts related to the notion of Frobenius categories, known as {\em Frobenius $n$-exact categories and Frobenius $n$-exangulated categories}, where $n$ is a positive integer. These notions have been introduced by Jasso \cite{ja} and Liu- Zhou \cite{lz}, respectively. \\ Extriangulated categories were introduced by Nakaoka and Palu \cite{np} by extracting those properties of $\Ext^1$ on exact categories and on triangulated categories that seem relevant
from the point of view of cotorsion pairs. As a higher dimensional analogue of this notion, Herschend, Liu, and Nakaoka \cite{hln, hln1} defined $n$-exangulated categories, for any positive integer $n$, giving also a reasonable
generalization of $n$-exact categories and $(n + 2)$-angulated categories, which has been introduced in \cite{gko}.\\
Assume that $(\mathcal{C}, \mathbb{E}, \frak{s})$ is an $n$-exangulated category. An object $P\in\mathcal{C}$ is projective, if for any distinguished $n$-exangle
$A_0\st{\alpha_0}\lrt A_1\st{\alpha_1}\lrt\cdots\rt A_{n}\st{\alpha_n}\lrt A_{n+1}\st{\delta}\dashrightarrow$ and any morphism $c\in\hom_{\mathcal{C}}(P, A_{n+1})$, there exists a morphism $b\in\hom_{\mathcal{C}}(P, A_n)$ such that $\alpha_n\circ b=c$. Moreover, $\mathcal{C}$ is said to have enough projective objects, provided that for any object $C\in\mathcal{C}$, there exists a distinguished $n$-exangle $$B\st{\alpha_0}\lrt P_1\st{\alpha_1}\lrt P_2\st{\alpha_2}\lrt\cdots\st{\alpha_{n-1}}\lrt P_n\st{\alpha_n}\lrt C\st{\delta}\dashrightarrow,$$ where $P_1,\cdots, P_n$ are projective. The notion of injective objects and also, having enough injective objects are defined dually. Now $\mathcal{C}$ is said to be a {\em Frobenius $n$-exangulated category}, if it has enough projectives, enough injectives, and the projectives coincide with the injectives, see \cite[Definition 3.2]{lz}.
The Frobenius $n$-exangulated category is called a {\em Frobenius $n$-exact category}, whenever $(\mathcal{C}, \mathbb{E}, \frak{s})$ is an $n$-exact category, see \cite[Definition 5.5]{ja} and \cite[Remark 3.3]{lz}.\\ It is worth noting that, in the case $n=1$ and $(\mathcal{C}, \mathbb{E}, \frak{s})$ is an $n$-exact category, the above-mentioned concepts agree with the classical notion of Frobenius categories, in particular, these coincide with 0-Frobenius categories in our sense, see Remark \ref{remexam}.
\end{rem}
}

\begin{s} Assume that $\C$ is an $n$-Frobenius category. Then, for any $k\geq 1$, a given object $N\in\C$ fits into conflations of length $k$;
$\syz^kN\rt P_{k-1}\rt\cdots\rt P_0\rt N$ and $N\rt P^1\rt\cdots\rt P^k\rt\syz^{-k}N$
such that $P_i, P^i$'s are $n$-projective, which will be called {\it unit conflations}. Also, $\syz^kN$ is said to be a $k$-th syzygy of $N$. Clearly, unit conflations are not uniquely determined. We denote the class of all unit conflations of length $k$ ending at $N$ (resp. beginning with $N$) by $\U_k(N)$ (resp. $\U^k(N)$). Unit conflations usually will be depicted by $\delta$.
\end{s}

\begin{rem}\label{pp1} (1) Assume that $k\geq 1$ and there exists a commutative diagram

\begin{equation}\label{11}
\begin{split}
{\footnotesize \xymatrix{\alpha:N \ ~\ar[r] \ \ \ar[d]_f& \ \ X_{k-1}\ar[r] \ \ \ar[d]_{h_{k-1}}& \cdots \ar[r]& \ \ X_0\ar[r] \ \ \ar[d]_{h_0} & \ M\ar[d]_g\\ \be:N' \ ~\ar[r] \ \ &\ \ Y_{k-1}\ar[r] \ \ &\cdots\ar[r]& \ \ Y_0\ar[r] \ \ & M',}}
\end{split}
\end{equation}
{where rows are conflations. Then we claim that $f\al=\be g$. The case $k=1$ is indeed \cite[Chapter VII, Lemma 1.1]{mit}. So assume that $k\geq 2$ and take the pull-back diagram

{\footnotesize \[\xymatrix{\be g:N'\ ~\ar[r] \ \ \ar@{=}[d]& \ \ Y_{k-1}\ar[r] \ \ \ar@{=}[d]& \cdots \ar[r]& \ \ Y_1\ar[r] \ \ \ar@{=}[d] & \ \ T\ar[r] \ \ \ar[d] & \ M\ar[d]_g\\ \beta:N' \ ~\ar[r] \ \ &\ \ Y_{k-1}\ar[r] \ \ &\cdots\ar[r]& \ \ Y_1\ar[r]& \ \ Y_0\ar[r] \ \ & M'.}\]}The universal property of the pull-back yields the existence of the commutative diagram {\footnotesize \[\xymatrix{\al:N\ ~\ar[r] \ \ \ar[d]_{f}& \ \ X_{k-1}\ar[r] \ \ \ar[d]_{h_{k-1}}& \cdots \ar[r]& \ \ X_1\ar[r] \ \ \ar[d]_{h_1} & \ \ X_0\ar[r] \ \ \ar[d] & \ M\ar@{=}[d]\\ \be g:N' \ ~\ar[r] \ \ &\ \ Y_{k-1}\ar[r] \ \ &\cdots\ar[r]& \ \ Y_1\ar[r]& \ \ T\ar[r] \ \ & M.}\]}Similarly, taking the push-out of $\al$ along $f$ and considering the latter diagram, the universal property of the push-out would give us the commutative diagram
{\footnotesize \[\xymatrix{f\al:N'\ ~\ar[r] \ \ \ar@{=}[d]& L\ar[r]\ar[d]& \ \ X_{k-2}\ar[r] \ \ \ar[d]_{h_{k-2}} & \cdots \ar[r] &\ \ X_1\ar[r] \ \ \ar[d]_{h_1} & \ \ X_0\ar[r] \ \ \ar[d] & \ M\ar@{=}[d]\\ \be g:N' \ ~\ar[r] \ \ &\ \ Y_{k-1}\ar[r] \ \ &\ \ Y_{k-2}\ar[r] \ \ &\cdots\ar[r]& \ \ Y_1\ar[r]& T\ar[r] \ \ & M.}\]}}Consequently, applying \cite[Chapter VII, Proposition 3.1]{mit} yields that $f\al=\be g$. In particular, {if $f$ (resp. $g$) is the identity morphism, then (\ref{11}) is a pull-back (resp. push-out) diagram. }\\ (2) Assume that $\ga\in\Ext^k_{\C}(C, A) $ and take two morphisms $h:C'\rt C$ and $l:A\rt A'$ in $\C$. It can be easily seen that there is a commutaive diagram
{\footnotesize \[\xymatrix{\gamma h: A \ \ ~\ar[r] \ \ \ar[d]_l& \ \ X_{k-1}\ar[r] \ \ \ar[d]& \cdots \ar[r]& \ \ X_0\ar[r] \ \ \ar[d] & \ C'\ar[d]_h\\ \ l\gamma:A' \ \ ~\ar[r] \ \ \ &\ \ Y_{k-1}\ar[r] \ \ &\cdots\ar[r]& \ \ Y_0\ar[r] \ \ & C.}\]}
So, as we have observed just above, $l(\ga h)=(l\ga)h$, see also \cite[Page 171, (2)]{mit}.
\end{rem}

\begin{s}
Set $\H:=\bigcup_{M,N\in\C}\hom_{\C}(M, N)$ and $\Ext^n:=\bigcup_{M,N\in\C}\Ext^n_{\C}(M, \syz^nN)$, where $\syz^nN$ runs over all the $n$-th syzygies of $N$. \\
Assume that $\ga\in\Ext^n_{\C}(M, \syz^n N)$, $a\in\hom_{\C}(M', M)$ and $b\in\hom_{\C}(\syz^nN, {\syz'}^nN)$. According to Remark \ref{pp1}, we have $b(\ga a)=(b\ga)a$. So $\Ext^n$ has an $\H$-bimodule structure.
\end{s}

\begin{rem}\label{zero}Assume that $f:M\rt N$ is a morphism in $\C$. So, for any $X\in\C$ and $k\geq 0$, one has the induced morphism $\Ext^k_{\C}(X, M)\rt\Ext^k_{\C}(X, N)$ mapping each $\ga$ to $f\ga$, the push-out of $\ga$ along $f$. Similarly, $\Ext^k_{\C}(N, X)\rt\Ext^k_{\C}(M, X)$ sends each object $\al$ to $\al f$, the pull-back along $f$. These morphisms can be interpreted as multiplication by $f$ from the left and the right, respectively. In this perspective, if there is no ambiguity, we denote the both morphisms again by $\f$. One should note that, in the case $k=0$, since $\Ext^0_{\C}(-, -)=\hom_{\C}(-, -)$, $f\ga$ and $\al f$ are just the composition of morphisms.
\end{rem}

In the sequel, we will see that any abelian category with non-zero $n$-projective objects, admits a non-trivial $n$-Frobenius subcategory. First, we state a definition.\\
\begin{dfn} An acyclic complex of $n$-projective objects $\mathbf{P^{\bullet}}:\cdots\lrt P^{i-1}\st{d^{i-1}}\lrt P^i\st{d^i}\lrt P^{i+1}\lrt\cdots$ is said to be {\it a complete resolution of $n$-projective objects}, if $\Ext^{n+1}_{\A}(\im d^i, Q)=0$, for any $i\in\Z$ and $Q\in n$-$\proj\A$.\\
Assume that $\I$ is a resolving subcategory of $n$-$\proj\A$, that is, in any conflation $P'\rt P\rt P''$, with $P''\in\I$, we have $P'\in\I$ if and only if $P\in\I$. Assume that $\C(\I)$ is the subcategory of $\A$ consisting of all objects $M$ which is a syzygy of a complete resolution of objects in $\I$, i.e., an acyclic complex $\mathbf{P^{\bullet}}:\cdots\lrt P^{i-1}\st{d^{i-1}}\lrt P^i\st{d^i}\lrt P^{i+1}\lrt\cdots,$ in $\I$ such that $\Ext^{n+1}_{\A}(\im d^i, Q)=0$, for any $i\in\Z$ and $Q\in\I$.
If $\I=n$-$\proj\A$, instead of $\C(n$-$\proj\A)$, we write $\C(\A).$ It is evident that $\C(\I)$ is a full subcategory of $\A$ containing all objects in $\I$ and it is closed under finite direct sums. Moreover, as the next proposition indicates, $\C(\I)$ is an exact category. First, we need a couple of preliminary lemmas. In the rest of this section, we assume that the abelian category $\A$ has non-zero $n$-projective objects.
\end{dfn}

We should mention that for a construction $\mathbf{P^{\bullet}}$, assume that $\C$ is an $n$-Frobenius category and $M\in\C$ is arbitrary. As $\C$ has enough $n$-(projective-injective) objects, one gets conflations $\syz M\rt P^0\rt M$ and $M\rt P^1\rt\syz^{-1}M$ in $\C$, where $P^0, P^1\in n$-$\proj\C$. Again, we may take conflations $\syz^2M\rt P^{-1}\rt\syz M$ and $\syz^{-1}M\rt P^2\rt\syz^{-2}M$ with $P^{-1},P^2\in n$-$\proj\C$. Now continuing this manner and splicing the resulting conflations together, yields the complete resolution of $n$-projective objects $\mathbf{P^{\bullet}}:\cdots\lrt P^{i-1}\st{d^{i-1}}\lrt P^i\st{d^i}\lrt P^{i+1}\lrt\cdots$.

\begin{lem}\label{000}{The following statements are satisfied:
\begin{enumerate}\item Let $Q\rt M\st{f}\rt N$ be a conflation in $\A$ such that $Q\in\I$. Then for any $X\in\C(\I)$ and $\ga\in\Ext^n_{\A}(X, N)$, there exists an object $\ga'\in\Ext^n_{\A}(X, M)$ such that $f\ga'=\ga$. \item Let $M\st{f}\rt N\rt Q$ be a conflation in $\A$ with $Q\in n$-$\proj\A$. Then for any $X\in\A$ and $\ga\in\Ext^n_{\A}(M, X)$, there exists an object $\ga'\in\Ext^n_{\A}(N,X)$ such that $\ga'f=\ga$.
\end{enumerate}}
\end{lem}
\begin{proof}We only prove the statement (1), the other one is obtained dually. Since $Q\in\I$ and $X\in\C(\I)$, $\Ext^{n+1}_{\A}(X, Q)=0$, and so, making use of \cite[Chapter VII, Theorem 5.1]{mit} forces
$\Ext^n_{\A}(X, M)\st{\f}\rt\Ext^n_{\A}(X, N)$ to be an epimorphism, where $\f=\Ext^n(X, f)$. Hence, there exists $\ga'\in\Ext^ n_{\A}(X, M)$ such that $f\ga'=\ga$, as needed.
\end{proof}

\begin{lem}\label{cok}The following assertions hold: \begin{enumerate}\item
Let $Q\rt M\st{f}\rt N$ be a conflation in $\A$ such that $Q\in\I$. If $N\in\C(\I),$ then so does $M$.\item Let $M\st{f}\rt N\rt Q$ be a conflation in $\A$ such that $Q\in\I$. If $M\in\C(\I)$, then so does $N$.\end{enumerate}
\end{lem}
\begin{proof}(1) Since $N\in\C(\I)$, by the definition, there exists a complete resolution of objects in $\I$; $\cdots\rt Q^{-1}\st{d^{-1}}\rt Q^0\st{d^0}\rt Q^{1}\st{d^{1}}\rt Q^{2}\rt\cdots,$ such that $N=\im d^0$. Take the conflation $\ga:N\rt Q^{1}\rt\cdots\rt Q^{n}\rt\syz^{-n}N$ in $\A$. {Clearly, $\syz^{-n}N\in\C(\I)$. So} applying Lemma \ref{000}(1), gives us an object $\ga'\in\Ext^n_{\A}(\syz^{-n}N, M)$ such that $f\ga'=\ga$. Namely, we have the following push-out diagram:
{\footnotesize\[\xymatrix{\ga':M~ \ \ \ar[r]\ar[d]_{f} & \ \ T\ar[d]\ar[r] \ \ & \ \ Q^{2}\ar[r]\ \ \ar@{=}[d] &\cdots \ar[r]& \ \ Q^{n}\ar@{=}[d]\ar[r] \ \ & \ \ \syz^{-n}N\ar@{=}[d]\\ \ga:N~\ar[r] \ \ & \ \ Q^{1}\ar[r]\ \ & \ \ Q^{2}\ar[r]\ \ & \cdots \ar[r] &\ \ Q^{n}\ar[r] \ \ & \ \ \syz^{-n}N.}\]} {Next take the following pull-back diagram: \[\xymatrix{L~\ar[r]\ar@{=}[d]& T'\ar[r]\ar[d]& M\ar[d]_{f}\\ L~\ar[r] & Q^0\ar[r]& N.}\]
Since $\I$ is closed under extensions, one may get that $T$ and $T'$ belong to $\I$. In particular, we obtain the complete resolution of {objects in $\I$,} $\cdots\rt Q^{-1}\rt T'\rt T\rt Q^{2}\rt\cdots$ such that $M=\im(T'\rt T)$, and then, $M\in\C(\I)$.}\\ (2) This is obtained by dualizing the argument given in the first assertion, so we skip it. Thus the proof is completed.
\end{proof}

\begin{prop}\label{proj1}The category $\C(\I)$ is closed under extensions and kernels of epimorphisms.
\end{prop}
\begin{proof}First let us prove that $\C(\I)$ is closed under extensions. To do this, assume that $M\st{f}\rt N\st{g}\rt K$ is a conflation in $\A$ such that $M, K\in\C(\I)$. We shall prove that $N\in\C(\I)$, as well. By the hypothesis, we may take a conflation $M\rt Q^{1}\rt\syz^{-1} M$ in $\A$ such that $Q^{1}\in\I$ and $\syz^{-1}M\in\C(\I)$. So considering the push-out diagram \[\xymatrix{M~\ar[r]\ar[d]_{f}& Q^{1}\ar[r]\ar[d]& \syz^{-1}M\ar@{=}[d]\\ N \ar[r] & T\ar[r] & \syz^{-1}M,}\] one obtains the conflation $Q^{1}\rt T\rt K$ in $\A$. As $K\in\C(\I)$, Lemma \ref{cok}(1) implies that $T\in\C(\I)$. This enables us to have the following commutative diagram: \[\xymatrix{N~\ar[r]\ar@{=}[d]& T\ar[r]\ar[d]& \syz^{-1}M\ar[d]_{f_1}\\ N~\ar[r] & P^{1}\ar[r]\ar[d]& G^1\ar[d]\\ & \syz^{-1}T~\ar@{=}[r] & \syz^{-1}T,}\] in which $P^{1}\in\I$ and $\syz^{-1}T\in\C(\I)$. As $\Ext^{n+1}_{\A}(\syz^{-1}M, Q)=0=\Ext^{n+1}_{\A}(\syz^{-1}T, Q)$ for any object $Q\in\I$, one deduces that $\Ext^{n+1}_{\A}(G^1, Q)=0$.
{Next consider a conflation $\syz^{-1}M\rt Q^2\rt\syz^{-2}M$, where $Q^2\in\I$ and $\syz^{-2}M\in\C(\I)$. Taking the push-out diagram \[\xymatrix{\syz^{-1}M~\ar[r]\ar[d]_{f_1}& Q^{2}\ar[r]\ar[d]& \syz^{-2}M\ar@{=}[d]\\ G^1 \ar[r] & T^1\ar[r] & \syz^{-2}M,}\] gives a conflation $Q^2\rt T^1\rt\syz^{-1}T$, in which $T^1\in\C(\I)$, thanks to Lemma \ref{cok}(1). Consequently, similar to the above, one may get the following commutative diagram: \[\xymatrix{G^1~\ar[r]\ar@{=}[d]& T^1\ar[r]\ar[d]& \syz^{-2}M\ar[d]_{f_2}\\ G^1~\ar[r] & P^{2}\ar[r]\ar[d]& G^2\ar[d]\\ & \syz^{-1}T^1~\ar@{=}[r] & \syz^{-1}T^1}\] such that $P^{2}\in\I$ and $\syz^{-1}T^1\in\C(\I)$. In particular,  the right-hand column, yields that $\Ext^{n+1}_{\A}(G^2, Q)=0,$ for any object $Q\in\I$.} Now, repeating this manner, gives rise to the existence of an acyclic complex $0\rt N\st{\epsilon}\rt P^{1}\st{d^{1}}\rt P^{2}\st{d^{2}}\rt\cdots,$ where each $P^i$ belongs to $\I$ and $\Ext^{n+1}_{\A}(\im d^i, Q)=0$ for any object $Q\in\I$ and $i\geq 1$. {Moreover, a dual argument gives us an acyclic complex, $\cdots\rt P^{-1}\st{d^{-1}}\rt P^0\st{\epsilon'}\rt N\rt 0$, where $P^i\in\I$ and $\Ext^{n+1}_{\A}(\im d^i, Q)=0$, for all $i\leq -1$. }Thus $N\in\C(\I)$, as required.\\ Next, we show that $\C(\I)$ is closed under kernels of epimorphisms. So for a given conflation $M\st{f}\rt N\st{g}\rt K$ in $\A$ with $N, K\in\C(\I)$, we have to show that $M\in\C(\I)$. Since $N\in\C(\I)$, there exists a commutative diagram \[\xymatrix{M~\ar[r]\ar[d]_{f}& Q^{1}\ar[r]\ar@{=}[d]& T\ar[d]\\ N \ar[r] & Q^{1}\ar[r] & \syz^{-1}N,}\] with $\syz^{-1}N\in\C(\I)$ and $Q^{1}\in\I$. So applying the snake lemma gives the conflation $K\rt T\rt \syz^{-1}N$ in $\A$. Since $\syz^{-1}N, K\in\C(\I)$, as we have already seen, $T\in\C(\I)$. Thus, taking the right half of a complete resolution $0\rt T\rt P^1\st{d^2}\rt P^2\rt\cdots$ of $T$, yields an acyclic complex $0\lrt M\lrt Q^{1}\st{d^{1}}\lrt P^{1}\st{d^{2}}\lrt\cdots$, whenever all objects, except $M$, lie in $\I$ and $\Ext^{n+1}_{\A}(\im d^i, \I)=0$, for all $i$.
{Next take the following pull-back diagram: \[\xymatrix{& \syz K\ar@{=}[r]\ar[d]& \syz K\ar[d]\\ M~\ar[r]\ar@{=}[d] & L\ar[r]\ar[d]& P\ar[d]\\ M\ar[r] & N~\ar[r] & K,}\] where $P\in\I$. Since $N, \syz K\in\C(\I)$, by the first assertion, the same is true for $L$. So considering the pull-back diagram \[\xymatrix{\syz L~\ar[r]\ar@{=}[d]& G\ar[r]\ar[d]& M\ar[d]\\ \syz L~\ar[r] & P'\ar[r]\ar[d]& L\ar[d]\\ & P~\ar@{=}[r] & P,}\] and using the fact that $\I$ is resolving, we infer that $G\in\I$. This, in conjunction with $\syz L$ being in $\C(\I)$, would guarantee the existence of a left resolution of objects in $\I$ for $M$; $\cdots\rt P^{-1}\st{d^{-1}}\rt P^0\st{d^0}\rt G\rt M\rt 0$, such that $\Ext^{n+1}_{\A}(\im d^i, \I)=0$ for all $i\leq 0$. Consequently, $M\in\C(\I)$,} as required.
\end{proof}

\begin{rem}\label{exact}Assume that $A\rt B\rt C$ is a conflation in $\C$. The same argument given in the proof of \cite[Chapter VII, Theorem 5.1]{mit} yields that for any object $X\in\C$
and any $n\geq 1$, there exists an exact sequence $$\Ext^{n-1}_{\C}(X, C)\rt\Ext^n_{\C}(X, A)\rt\Ext^n_{\C}(X, B)\rt\Ext^n_{\C}(X, C)\rt\Ext^{n+1}_{\C}(X, A).$$ Also, a dual argument gives us the exact sequence
$$\Ext^{n-1}_{\C}(A, X)\rt\Ext^n_{\C}(C, X)\rt\Ext^n_{\C}(B, X)\rt\Ext^n_{\C}(A, X)\rt\Ext^{n+1}_{\C}(C, X).$$
\end{rem}

\begin{theorem}\label{subcat}$\C(\I)$ is an $n$-Frobenius subcategory of $\A$.
\end{theorem}
\begin{proof}First  note that by Proposition \ref{proj1}, $\C(\I)$ is an exact category. In view of the definition of $\C(\I)$, we only need to show that any $n$-injective object of $\C(\I)$ is also $n$-projective over $\C(\I)$ and vice versa. Take an object $N\in n$-$\inj\C(\I)$ and consider unit conflations $\delta_N:N\rt P^{1}\rt\cdots\rt P^{n}\lrt \syz^{-n}N$ and $\syz^{-n}N\st{h}\rt P^{n+1}\rt\syz^{-n-1} N$, where $P^i\in\I$, for any $i$. Since $\syz^{-n-1}N\in\C(\I)$, $\Ext^{n+1}_{\C(\I)}(\syz^{-n-1}N,N)=0$, there exists an object $\ga\in\Ext^n_{\C(\I)}( P^{n+1},N)$ such that $\ga h=\delta_N$. Namely, we have the following pull-back diagram:
{\footnotesize\[\xymatrix{\delta_N \ :N~ \ \ \ \ar[r] \ \ \ar@{=}[d] & \ \ P^{1}\ar@{=}[d]\ar[r]\ \ & \ \ \cdots \ar[r]\ \ &P^{n-1}\ar@{=}[d] \ar[r]& \ \ P^{n}\ar[d]\ar[r] \ \ & \ \ \syz^{-n}N\ar[d]^{h}\\ \ga : \ \ \ N~ \ \ \ar[r] \ \ \ & \ \ P^{1}\ar[r]\ \ & \ \ \cdots\ar[r] \ \ & P^{n-1} \ar[r] & \ \ H\ar[r] &P^{n+1}\ \ .}\]} As $H\in\C(\I)$, there is an inflation $H\rt Q$ with $Q\in\I$. So by taking the conflation $L\rt H\rt P^{n+1}$, one gets the following commutative diagram: \[\xymatrix{L~\ar[r]\ar@{=}[d]& H\ar[r]\ar[d]&P^{n+1} \ar[d]^{g}\\ L~\ar[r] & Q\ar[r]& \syz^{-1}L.}\] In particular, we have the following pull-back diagram of unit conflations:
\begin{equation}\label{xy}
\begin{split}
{\footnotesize \xymatrix{\delta_N : N~ \ \ \ \ar[r] \ \ \ar@{=}[d] & \ \ P^{1}\ar@{=}[d] \ \ \ar[r] \ \ & \ \ \cdots\ar[r] \ \ &P^{n-1} \ar[r] \ar@{=}[d]& \ \ P^{n}\ar[d]\ar[r] \ \ & \syz^{-n}N\ar[d]^{gh}\\ \beta : \ \ \ N~ \ \ \ar[r]\ \
\ & \ \ P^{1}\ar[r]\ \ & \cdots\ar[r]\ \ & P^{n-1} \ar[r] & \ \ Q\ar[r] \ \ &{\syz}^{-1}L}}
\end{split}
\end{equation}
One should note that, according to our notation, $\syz^{-1}L={\syz'}^{-n}N$.
Hence, for any object $X\in\C(\I)$, there is a commutative square{\footnotesize \[\xymatrix{\Ext_{\C(\I)}^{n+1}(N, X)~\ar[r]^{\cong}\ar@{=}[d] & \Ext_{\C(\I)}^{2n+1}({\syz'}^{-n}N, X)\ar[d]^{\g\h}\\ \Ext_{\C(\I)}^{n+1}(N, X)~\ar[r]^{\cong} & \Ext_{\C(\I)}^{2n+1}({\syz}^{-n}N, X).}\]}As $gh$ factors through $P^{n+1}$, the right column is zero, and then, $\Ext_{\C(\I)}^{n+1}(N, X)=0$, meaning that $N\in n$-$\proj\C(\I)$. Since the converse is obtained dually, we skip it. So the proof is finished.
\end{proof}

Assume that $R$ is a commutative noetherian ring. A finitely generated $R$-module $M$ is said to be of $G$-dimension zero if it is a syzygy of a complete resolution of projectives. The $G$-dimension of a finitely generated $R$-module $N$, $G$-$\dim_RN$, is the length of a shortest resolution of $N$ by $G$-dimension zero modules. If there is no such a resolution of finite length, then we write $G$-$\dim_RN=\infty$. This invariant has been defined by Auslander and Bridger \cite{ab} and provides a refinement of the projective dimension of a module. The category of all modules of $G$-dimension zero (resp. of finite $G$-dimension), is denoted by $\G$ (resp. $\G^{<\infty}$).

\begin{example}\label{ex1} Assume that $(R, \m)$ is a $d$-dimensional commutative noetherian local ring. {In view of the Auslander-Buchsbaum formula, any module of finite projective dimension, is $d$-projective. Set $\I:=d$-$\proj\md R$. Take an arbitrary object $M\in\G^{<\infty}$. So by \cite[Proposition 1.2]{et} (see also \cite[Lemma 2.17]{cfh}), there is a short exact sequence $0\rt M\rt P\rt X\rt 0$ in which $P\in\I$ and $X\in\G$. All of these facts would imply that $M\in\C(\I)$. Moreover, it is standard to see that any object in $\C(\I)$ lies in $\G^{<\infty}$. Hence by Theorem \ref{subcat}, $\G^{<\infty}$ is a $d$-Frobenius subcategory of $\md R$. In particular, if $R$ is Gorenstein, (i.e. $\id_RR$ is finite), then $\md R$ is indeed a $d$-Frobenius category. It should be noted that the category $\G$, is a $0$-Frobenius subcategory of $\md R$.}
\end{example}

In the sequel, we explore more examples of categories that admit $n$-Frobenius subcategories, for some non-negative integer $n$.
\begin{example}\label{ex123}
(1) According to \cite[Lemma 1.12]{or}, {the category of locally free sheaves of finite rank, $\L$, is a subcategory of $n$-$\proj\coh(\x)$, for some integer $n\geq 0$. So the category $n$-$\proj\coh(\x)$ is non-trivial.} Thus Theorem \ref{subcat} yields that $\coh(\x)$ has a non-trivial $n$-Frobenius subcategory, that we denote it by $\C(\x)$.\\
(2) Assume that $\qcoh(\x)$ is the category of quasi-coherent sheaves over $\x$ and $\fl\x$ is its subcategory of flats. Following the argument given in the proof of \cite[Lemma 1.12]{or}, gives rise to the existence of an integer $n\geq 0$ such that for any object $\F\in\fl\x$, $\Ext^{n+1}(\F, \Q)=0$, for all quasi-coherent sheaves $\Q$. Namely, $\fl\x$ is a subcategory of $n$-$\proj\qcoh(\x)$. Thus $\qcoh(\x)$ has enough $n$-projective objects, as it has enough flat sheaves, see \cite[Corollary 3.21]{mu}. Hence, Theorem \ref{subcat} implies that $\qcoh(\x)$ admits a non-trivial $n$-Frobenius subcategory $\C(\x)$. In addition, as $\fl\x$ is a resolving subcategory of $n$-$\proj\qcoh(\x)$, another use of Theorem \ref{subcat} yields that $\C(\fl\x)$ is also an $n$-Frobenius subcategory of $\qcoh(\x)$.\\

\end{example}

Assume that $\I$ is a resolving subcategory of $n$-$\proj\A$. In view of the definition of $\C(\I)$, we have $\I\subseteq n$-$\proj\C(\I)$. The next couple of results, provide examples in which the equality holds.

\begin{prop}\label{cf}Keeping the notation above, the equality $n$-$\proj\C(\fl\x)=\fl\x$ holds.
\end{prop}
\begin{proof}Since $\C(\fl\x)$ is an $n$-Frobenius subcategory of $\qcoh(\x)$, it suffices to prove that any object of $n$-$\inj\C(\fl\x)$ is flat. Take an arbitrary object $N\in n$-$\inj\C(\fl\x)$. Consider the unit conflations; $\delta:N\rt P^{1}\rt\cdots\rt P^{n}\rt\syz^{-n}N$ and $\syz^{-n}N\st{h}\rt P^{n+1}\rt\syz^{-n-1}N$, where $P^i\in\fl\x$, for any $i$. So, {according to the proof of Theorem \ref{subcat}, we get the diagram \ref{xy}.} In particular, for any object $X\in\qcoh(\x)$, we will have the following commutative square: {\footnotesize\[\xymatrix{\Tor_{n+1}(X, \syz^{-n}N)~\ar[r]^{\cong}\ar[d]_{\Tor_1(X, gh)} & \Tor_{1}(X, N)\ar@{=}[d]\\ \Tor_{n+1}(X, {\syz'}^{-n}N)~\ar[r]^{\cong} & \Tor_{1}(X, N).}\]}As $gh$ factors through the flat sheaf $P^{n+1}$, the left column is zero, implying that $\Tor_1(X, N)=0$, and then $N$ is flat, as needed.
\end{proof}

\begin{prop}\label{locally}For an integer $n\geq 0$, $\C(\L)$ is an $n$-Frobenius subcategory of $\coh(\x)$ with $n$-$\proj\C(\L)=\L$.
\end{prop}
\begin{proof} As we mentioned in Example \ref{ex123}(1), $\L$ is a subcategory of $n$-$\proj\coh(\x)$, which is evidently resolving. So by Theorem \ref{subcat}, $\C(\L)$, is an $n$-Frobenius subcategory of $\coh(\x)$. Moreover, Proposition \ref{cf} and the fact that every coherent flat sheaf is locally free, leads us to deduce that $n$-$\proj\C(\L)=\L$. So the proof is completed.
\end{proof}
{
\begin{rem}\label{chf}According to Example \ref{ex123}(2), there exists an integer $t\geq 0$ such that for any $F\in\fl\x$, we have $F\in t$-$\proj\fl\x$. On the other hand, it is known that there is an integer $k\geq 0$ such that for any object $F\in\fl\x$, one has an exact sequence $0\rt F\rt C^0\rt\cdots\rt C^k\rt 0$, where each $C^i$ is cotorsion flat, implying that $F\in k$-$\inj\fl\x$, because each $C^i$ lies in $\inj\fl\x$. So using the same method appeared in \cite[page 28]{ha1}, we may deduce that any complex of the form $\cdots\rt 0\rt F\st{1}\rt F\rt 0\rt\cdots$ with $F$ flat, is an $n$-projective and an $n$-injective object of $\ch(\fl\x)$, for some integer $n$. Here $\ch(\fl\x)$ stands for the category of complexes of flats. We let $\J$ denote the subcategory consisting of all contractible complexes of flats. {One should note that any object of $\J$ is a direct sum of complexes of the form $\cdots\rt 0\rt F\st{1}\rt F\rt 0\rt\cdots$ with $F$ flat, and so, it belongs to $n$-$\proj\ch(\fl\x)$ as well as to $n$-$\inj\ch(\fl\x)$.} Assume that $F^{\bullet}:\cdots\rt F^{-1}\rt F^0\rt F^{1}\rt\cdots$ is an arbitrary object of $\ch(\fl\x)$. So, according to the short exact sequence of complexes $0\rt F^{\bullet}\rt\co(1_{F^{\bullet}})\rt F^{\bullet}[1]\rt 0$, we infer that $F^{\bullet}$ is embedded in (and also a homomorphic image of) an object in $\J$. It is evident that $\J$ is a subcategory of $\ch_{\pp}(\fl\x)$, the category of flat complexes. The result below indicates that $\ch(\fl\x)$ is an $n$-Frobenius category with $n$-$\proj\ch(\fl\x)=\ch_{\pp}(\fl\x)$.

\end{rem}
\begin{theorem}\label{fp}For a non-negative integer $n$, $\ch(\fl\x)$ is an $n$-Frobenius category, with $n$-$\proj\ch(\fl\x)=\ch_{\pp}(\fl\x)$.
\end{theorem}
\begin{proof}{According to Remark \ref{chf}, $\ch(\fl\x)$ has enough $n$-projective and $n$-injective objects. So it remains to examine the validity of the equalities $n$-$\proj\ch(\fl\x)=\ch_{\pp}(\fl\x)=n$-$\inj\ch(\fl\x)$. In this direction, first we show that $n$-$\proj\ch(\fl\x)=n$-$\inj\ch(\fl\x)$.
Take an arbitrary object $N^\bullet\in n$-$\inj\ch(\fl\x)$. By using Remark \ref{chf}, one may have the conflations $N^{\bullet}\rt {P^{\bullet}}^1\rt\cdots\rt {P^{\bullet}}^n\rt\syz^{-n}N^{\bullet}$ and $\syz^{-n}N^{\bullet}\st{h^{\bullet}}\rt {P^{\bullet}}^{n+1}\rt\syz^{-n-1}N^{\bullet}$, where ${P^{\bullet}}^i\in\J$, for any $i$. Since $\syz^{-n-1}N^{\bullet}\in\ch(\fl\x),$ we have $\Ext^{n+1}_{\ch}(\syz^{-n-1}N^{\bullet}, N^{\bullet})=0$.
So, the argument given in the proof of Theorem \ref{subcat}, gives us a diagram similar to \ref{xy}. Take an arbitrary object $X^{\bullet}\in\ch(\fl\x)$. As for any $i>n$, $\Ext^i_{\ch}({P^{\bullet}}^j, X^{\bullet})=0$, applying the functor $\Ext_{\ch}(-, X^{\bullet})$, gives rise to the following commutative square: {\footnotesize \[\xymatrix{\Ext_{\ch}^{i}(N^{\bullet}, X^{\bullet})~\ar[r]^{\cong}\ar@{=}[d] & \Ext_{\ch}^{n+i}({\syz'}^{-n}N^{\bullet}, X^{\bullet})\ar[d]^{\g^{\bullet}\h^{\bullet}}\\ \Ext_{\ch}^{i}(N^{\bullet}, X^{\bullet})~\ar[r]^{\cong} & \Ext_{\ch}^{n+i}({\syz}^{-n}N^{\bullet}, X^{\bullet}).}\]}Since $g^{\bullet}h^{\bullet}$ factors through an object of $\J$, the right column will be zero, and then, $\Ext_{\ch}^{i}(N^{\bullet}, X^{\bullet})=0$ for any $i>n$, meaning that $N^{\bullet}\in n$-$\proj\ch(\fl\x)$. The dual method indicates that if $N^{\bullet}\in n$-$\proj\ch(\fl\x)$, it will belong to $n$-$\inj\ch(\fl\x)$. Next, assume that $N^\bullet\in n$-$\inj\ch(\fl\x)$ and consider the latter conflations. Take an arbitrary object $X\in\qcoh(\x)$. As ${P^{\bullet}}^j$, for any $j$, is contractible, $\H_i({P^{\bullet}}^j\otimes_{\O_{\x}}X)=0$ for all $i\in \Z$. Hence, one may obtain the following commutative square: \[\xymatrix{\H_{n+i}(\syz^{-n}N^\bullet\otimes_{\O_{\x}} X)~\ar[r]^{\cong}\ar[d] & \H_{i}(N^\bullet\otimes_{\O_{\x}} X)\ar@{=}[d]\\ \H_{n+i}({\syz'}^{-n}N^\bullet\otimes_{\O_{\x}} X)~\ar[r]^{\cong} & \H_{i}(N^\bullet\otimes_{\O_{\x}} X).}\] Since $g^{\bullet}h^{\bullet}$ factors through ${P^{\bullet}}^{n+1}$, the left column will be zero, implying that $\H_i(N^\bullet\otimes_{\O_{\x}}X)=0$, for all $i\in\Z$. Since $X$ is arbitrary, we conclude that $N^\bullet\in\ch_{\pp}(\fl\x)$.
Conversely, assume that $N^\bullet\in\ch_{\pp}(\fl\x)$. By \cite[Proposition 2.6]{hs}, there is a short exact sequence, $0\lrt N^\bullet\lrt {C^\bullet}^0\lrt {P^\bullet}^1\lrt 0$ in $\ch(\fl\x)$ with ${P^\bullet}^1\in\ch_{\pp}(\fl\x)$ and
${C^\bullet}^0$ is dg-cotorsion, that is, ${C^\bullet}^0\in\ch_{\pp}(\fl\x)^{\bot}$.
Thus ${C^\bullet}^0\in\ch_{\pp}(\fl\x)$, as well. Repeating this manner, yields the existence of an exact sequence $0\lrt N^\bullet\lrt {C^\bullet}^0\lrt {C^\bullet}^1\lrt\cdots\lrt {C^\bullet}^{k-1}\lrt {C^\bullet}^k\lrt 0$ in $\ch(\fl\x)$, with $k=\dim\x$. Since ${C^\bullet}^i$, for any $i$, is a pure acyclic complex of cotorsion flats, it is contractible by \cite[Corollary 3.1.2]{hs}, and in particular, it will belong to $n$-$\proj\ch(\fl\x)$. Hence, it is easily seen that $N^{\bullet}\in n$-$\proj\ch(\fl\x)$. So the proof is finished.}

\end{proof}

\begin{rem}\label{remfin}As we mentioned in Example \ref{ex123}(1), $\L$ is a subcategory of $n$-$\proj\coh(\x)$, for some integer $n$. On the other hand, since any locally free sheaf is flat and, as observed in Remark \ref{chf}, each flat sheaf lies in $n$-$\inj\L$, we conclude that $\L$ will be also a subcategory of $n$-$\inj\L$. So, similar to Remark \ref{chf}, we infer that any complex of the form $\cdots\rt 0\rt L\st{1}\rt L\rt 0\rt\cdots$, with $L$ locally free, is an $n$-projective and $n$-injective over $\ch(\L)$, the category of complexes of locally free sheaves of finite rank. Suppose that $\J$ is the subcategory consisting of all contractible complexes of locally free sheaves of finite rank. Again, similar to Remark \ref{chf}, one may observe that any object of $\ch(\L)$ can be embedded in an object of $\J$, as well as, it is a homomorphic image of an object in $\J$. The same argument given in the proof of Theorem \ref{fp}, clarifies that the subcategory $\ch_{\pp}(\L)$ consisting of all acyclic complexes of locally free sheaves with locally free kernels, forms $n$-projective objects of $\ch(\L)$. Precisely, we have the next interesting result.
\end{rem}
\begin{theorem}$\ch(\L)$ is an $n$-Frobenius category, for some integer $n$. Moreover, $n$-$\proj\ch(\L)=\ch_{\pp}(\L)$.
\end{theorem}

\section{Quasi-invertible morphisms}

Assume that $\C$ is an $n$-Frobenius category. In this section, we will show that a morphism $f$ acts as invertible on $\Ext_{\C}^{n+1}$ from the left and the right, simultaneously. A morphism satisfying this condition will be called a quasi-invertible morphism in $\C$.

In the remainder of this paper, unless otherwise specified, by a conflation of length $t$, we mean a conflation of length $t$ in $\C$. Also, if there is no ambiguity, we drop the ``of length $t$''. Furthermore, instead of $\Ext^i_{\C}(-,-)$, we write $\Ext^i(-,-)$.

\begin{lem}\label{101}(1) Let $N\st{f}\rt X$ and $N\st{g}\rt X'$ be two morphisms in $\C$ such that $f$ or $g$ is an inflation. Then $N\st{[f~~g]^t}\lrt X\oplus X'$ is also an inflation.\\ (2) Let $X\st{f}\rt N$ and $X'\st{g}\rt N$ be two morphisms in $\C$ such that $f$ or $g$ is a deflation. Then $X\oplus X'\st{[f~~g]}\lrt N$ is a deflation.
\end{lem}
\begin{proof}Let us prove the first assertion. The second one is obtained dually.
Without loss of generality, we may assume that $f$ is an inflation. Consider the following commutative diagram: {\footnotesize\[\xymatrix{ & X'\ar@{=}[r]\ar[d]& X'\ar[d]\\ N~\ar[r]^{{{\tiny {\left[\begin{array}{ll} f \\ g \end{array} \right]}}}}\ar@{=}[d] & X\oplus X' \ar[r]\ar[d]_{[1~0]}&L'\ar[d]_{h}\\ N\ar[r]^{f}&X\ar[r]& L.}\]}Since the bottom row is a conflation and $\C$ is closed under extensions, we infer that $L'\in\C$, and so, $h$ is a morphism in $\C$. Now, as $\C$ is closed under pull-back, the middle row will be a conflation, giving the desired result.
\end{proof}

\begin{dfn}Assume that $f:M\rt N$ is a morphism in $\C$ and $i\geq 0$ is an integer. We say that $f$ acts as invertible on $\Ext^i$ from the left (resp. right), if for any $X\in\C$, $\Ext^i(X, M)\st{\f}\rt\Ext^i(X, N)$ (resp. $\Ext^i(N, X)\st{\f}\rt\Ext^i(M, X)$) is an isomorphism.
\end{dfn}

{
\begin{lem}\label{conf}
Let $M\st{f}\rt N$ be a morphism in $\C$ acting as invertible on $\Ext^{n+1}$ from the left or the right. Assume that $Q\st{\pi}\rt N$ is a deflation and $M\st{l}\rt P$ is an inflation such that $P, Q\in n$-$\proj\C$. Then
\begin{enumerate}\item $M\oplus Q\st{[f~~\pi]}\rt N$ is a deflation such that its kernel is $n$-projective.\item $M\st{[f~~l]^{t}}\rt N\oplus P$ is an inflation such that its cokernel is $n$-projective.\end{enumerate}
\end{lem}
\begin{proof}We only deal with the case $f$ acts as invertible on $\Ext^{n+1}$ from the left. The other case can be treated similarly. Since $M\st{l}\rt P$ is an inflation, by Lemma \ref{101}(1), $M\st{h}\rt N\oplus P\rt L$ is a conflation in $\C$, where $h={{{\tiny {\left[\begin{array}{ll} f \\ l \end{array} \right]}}}}$. {we shall prove that $L\in n$-$\proj\C$. Take an arbitrary object $X\in\C$. By applying the functor $\Ext(X, -)$ to this conflation and using Remark \ref{exact}, } we obtain the long exact sequence{\footnotesize $$\cdots\rt\Ext^{i}(X, M)\st{\h}\rt\Ext^{i}(X, N\oplus P)\rt\Ext^{i}(X, L)\rt\Ext^{i+1}(X, M)\st{\h}\rt\Ext^{i+1}(X, N\oplus P)\rt\cdots.$$}Since $f$ acts as invertible on $\Ext^{n+1}$, it acts as invertible on $\Ext^i$ for all $i>n$. {To see this, it suffices to prove the case only for $n+2$. Consider a conflation $\syz X\rt Q\rt X$, in which $Q\in n$-$\proj\C$. So another use of Remark \ref{exact} together with the $n$-projectivity of $Q$, gives us the following commutative square: {\footnotesize \[\xymatrix{\Ext^{n+1}(\syz X, M) \ar[r]^{\cong}\ar[d]_{\Ext^{n+1}({\syz X}, f)}& \Ext^{n+2}(X, M)\ar[d]_{\Ext^{n+2}(X, f)}\\ \Ext^{n+1}(\syz X, N)\ar[r]^{\cong}&\Ext^{n+2}(X, N).}\]}By our hypothesis, the left column is an isomorphism, and so, the same will be true for the right one. Namely, $f$ acts as invertible on $\Ext^{n+2}$.} Since $P\in n$-$\proj\C$, $l\Ext^{i}=0$ for all $i>n$, implying that $f\Ext^{i}=h\Ext^{i}$. In particular, $h$ acts as invertible on $\Ext^{i}$, and then, $\Ext^{i}(X, L)=0$ for all $i>n$. This means that $L\in n$-$\inj\C$, and so, it belongs to $n$-$\proj\C$, {giving the first assertion. For the second one,}
one may apply Lemma \ref{101}(2) to obtain the conflation $L'\rt M\oplus Q\st{[f~~\pi]}\rt N$. So, repeating the above method, yields that $\Ext^{i}(X, L')=0$ for all $i>n+1$. Now, as $\C$ is $n$-Frobenius, it is easily seen that $L'\in n$-$\proj\C$. So, the proof is finished.
\end{proof}

\begin{cor}\label{lr}Let $f:M\rt N$ be a morphism in $\C$. Then $f$ acts as invertible on $\Ext^{n+1}$ from the left if and only if it acts as invertible from the right.
\end{cor}
\begin{proof}Assume that $f$ acts as invertible on $\Ext^{n+1}$ from the left. In view of Lemma \ref{conf}, there exists a conflation $M\st{h}\rt N\oplus P\rt P'$ in $\C$, where $P, P'\in n$-$\proj\C$. Suppose that $X\in\C$ is arbitrary. Applying the functor $\Ext(-, X)$ to this conflation, yields the long exact sequence $$\cdots\rt\Ext^{i}(P', X)\rt\Ext^{i}(N\oplus P, X)\st{\h}\rt\Ext^{i}(M, X)\rt\Ext^{i+1}(P', X)\rt\cdots.$$ As $\Ext^{i}(P', X)=0$, for any $i>n$, $h$ (and so, $f$) will act as invertible on $\Ext^{n+1}$ from the right. Since the sufficiency can be shown in a dual manner, we ignore it. So the proof is finished.
\end{proof}
}

\begin{lem}\label{epi}
Let $M\st{f}\rt N$ be a morphism in $\C$ such that $\ker f, \cok f\in n$-$\proj\C$. Assume that $Q\st{\pi}\rt N$ is a deflation and $M\st{i}\rt P$ is an inflation such that $P, Q\in n$-$\proj\C$. Then
\begin{enumerate}\item
$M\oplus Q\st{[f~~\pi]}\lrt N$ is a deflation such that its kernel is $n$-projective.\item $M\st{[f~~i]^{t}}\lrt N\oplus P$ is an inflation such that its cokernel is $n$-projective.
\end{enumerate}
\end{lem}
\begin{proof}Let us prove only the first assertion. The second one is obtained dually.
By the hypothesis, there exist conflations $P'\rt M\st{h}\rt L$ and $L\st{g}\rt N\rt P''$ such that $gh=f$ and $P', P''\in n$-$\proj\C$. Consider the following commutative diagram: {\footnotesize\[\xymatrix{T~\ar[r]\ar[d]_{\varphi}& Q\ar[r]^{l\pi}\ar[d]_{\pi}& P''\ar@{=}[d]\\ L~\ar[r]^{g} & N\ar[r]^l& P'',}\]}where $\varphi$ is an induced map. One should note that since $l$ and $\pi$ are deflation, $l\pi$ is so. Thus, the top row is also a conflation. Since $Q, P''\in n$-$\proj\C$, the same is true for $T$.
In particular, we have the conflation $T\st{[-\varphi~~\alpha]^t}\lrt L\oplus Q\st{[g~~\pi]}\lrt N$. Now consider the following pull-back diagram: {\footnotesize\[\xymatrix{P'~\ar[r]\ar@{=}[d]& T'\ar[r]\ar[d]& T\ar[d]\\ P'~\ar[r] & M\oplus Q\ar[r]^{u}\ar[d]& L\oplus Q\ar[d]\\ & N~\ar@{=}[r] & N,}\]}where $u={\tiny {\left[\begin{array}{ll} h & 0 \\ 0 & {1} \end{array} \right]}}$. Evidently, $T'$ is $n$-projective, because
$P',T\in n$-$\proj\C$, and then, $T'\rt M\oplus Q\st{[f~~\pi]}\rt N$ is the desired conflation. So the proof is finished.
\end{proof}
As a consequence of Lemma \ref{epi} and the proof of Corollary \ref{lr}, we include the next result.
\begin{cor}\label{is}Let $f:M\rt N$ be a morphism in $\C$ such that $\ker f, \cok f\in n$-$\proj\C$. Then $f$ acts as invertible on $\Ext^{n+1}$.
\end{cor}

\begin{dfn}We say that a given morphism $f$ in $\C$ is quasi-invertible, provided that $f$ acts as invertible on $\Ext^{n+1}$. The class of all quasi-invertible morphisms will be denoted by $\si$.
\end{dfn}

\begin{rem}\label{rems}Assume that $f:M\rt N$ is a morphism in $\si$. As we have seen in the proof of Lemma \ref{conf}, taking an inflation (resp. a deflation) $g: M\rt P$ (resp. $g: P\rt N$), yields an inflation (resp. a deflation) $M\st{[f~~g]^t}\lrt N\oplus P$ (resp. $M\oplus P\st{[f~~g]}\lrt N$) such that its cokernel (resp. kernel) is $n$-projective. Moreover, since $\Ext^{n+1}(-, Q)=0=\Ext^{n+1}(Q, -)$, for any $Q\in n$-$\proj\C$, we may deduce that the maps $f$, ${\tiny {\left[\begin{array}{ll} f \\ g \end{array} \right]}}$ and $[f~~g]$ act identically on $\Ext^{n+1}$. So, if $f$ acts as invertible on $\Ext^{n+1}$, without loss of generality, we may further assume that $f$ is an inflation or a deflation with cokernel and kernel $n$-projective, respectively.
\end{rem}

The result below reveals that being a unit conflation is stable under the pull-back and push-out along morphisms in $\si$.

\begin{lem}\label{unit}Let $\ga\in\Ext^k(N, \syz^k N)$ with $k\geq 1$. Let $a:X\rt N$ and $b:\syz^kN\rt Y$ be two morphisms in $\si$. Then the following assertions hold:
\begin{enumerate}\item $\ga$ is a unit conflation if and only if $\ga a$ is so. \item $\ga$ is a unit conflation if and only if $b\ga$ is so.
\end{enumerate}
\end{lem}
\begin{proof}We only prove the first assertion. The second one is obtained dually. First one should note that, by the definition of pull-back diagram, without loss of generality, we may assume that $k=1$. As $a\in\si$ , by Lemma \ref{conf}, there exists a conflation $Q\rt X\oplus P\st{[a~~\pi]}\rt N$, where $P, Q\in n$-$\proj\C$. So taking the following pull-back diagram: {\footnotesize \[\xymatrix{\ga [a~~\pi]: \syz N~\ar[r]\ar@{=}[d]& H\ar[r]\ar[d]& X\oplus P\ar[d]_{[a~~\pi]}\\ \ga: \syz N \ar[r] & T\ar[r] & N,}\]}gives rise to the conflation, $Q\rt H\rt T$. We show that $H$ and $T$ are $n$-projective, simultaneously. If $T$ is $n$-projective, then the same is true for $H$, because $n$-$\proj\C$ is closed under extensions. Conversely, assume that $H$ is $n$-projective. As $\C$ is an $n$-Frobenius category, it is easily seen that the class $n$-$\proj\C$ is closed under cokernels of monomorphisms, implying that $T$ is $n$-projective. This means that the conflation $\ga$ is unit if and only if $\ga[a~~\pi]$ is so. Next considering the following pull-back diagram: {\footnotesize\[\xymatrix{\be:\syz N~\ar[r]\ar@{=}[d]& L\ar[r]\ar[d]& X\ar[d]_{{{\tiny {\left[\begin{array}{ll} 1\\ 0 \end{array} \right]}}}}\\ \ga[a~~\pi]:\syz N~\ar[r] & H\ar[r]\ar[d]& X\oplus P\ar[d]\\ & P~\ar@{=}[r] & P,}\]}we conclude that $\ga[a~~\pi]$ and $\be$ are unit conflation, simultaneously. Finally, the equality $\be=(\ga[a~~\pi]){{{\tiny {\left[\begin{array}{ll} 1\\ 0 \end{array} \right]}}}}=\ga a$, completes the proof.
\end{proof}

We close this section with the following result.
\begin{lem}\label{sig}
Let $f: M\rt N$ be an inflation or a deflation in $\C$. Then the following assertions hold:
\begin{enumerate}
\item If there exists $\delta_N\in\U_n(N)$ such that $\delta_Nf\in\U_n(M)$, then $f$ lies in $\si$.
\item If there exists $\delta_M\in\U^n(M)$ such that $f\delta_M\in\U^n(N)$, then $f$ belongs to $\si$.
\end{enumerate}
\end{lem}
\begin{proof}
We only prove the first assertion. The second one is obtained dually. Without loss of generality,
we assume that $f$ is an inflation. By the definition of the pull-back diagram, we may assume that $n=1$. By the hypothesis, there exists a pull-back diagram
{\footnotesize\[\xymatrix{\syz N~\ar[r] \ar@{=}[d]& T\ar[r]\ar[d]_{h}& M\ar_{f}[d]\\ \delta_N:\syz N~\ar[r] & P\ar[r]& N}\]}where $P,T\in n$-$\proj\C$. Now since $h$ is an inflation, $\cok h$ is $n$-projective, and so, the same will be true for $\cok f$. Now Corollary \ref{is} completes the proof.
\end{proof}

\section{$\p$-subfunctor of $\Ext^n$}
Assume that $\C$ is an $n$-Frobenius category. This section aims to study a subfunctor of $\Ext^n$ consisting of all conflations arising as a pull-back along morphisms ending at $n$-projective objects that we call a $\p$-subfunctor of $\Ext^n$. We begin with the following useful observation.

\begin{rem}\label{use}Assume that $X, Z$ are arbitrary objects of $\C$. Consider the unit conflations $Z\rt P\rt\syz^{-1}Z$ and $\syz X\rt Q\rt X$. So, we will have the following commutative diagram with exact rows and columns:
\[\xymatrix{
\Ext^n(P, Q)~\ar[r]\ar[d]& \Ext^n(Z, Q)\ar[r]\ar[d]_{\beta}& \Ext^{n+1}(\syz^{-1}Z, Q)\ar[d]\ar[r] &0\\ \Ext^n(P, X)~\ar[r]^{\alpha}\ar[d]& \Ext^n(Z, X)\ar[r]^{\psi}\ar[d]_{\varphi}& \Ext^{n+1}(\syz^{-1}Z, X)\ar[d]_{\eta}\ar[r] &0\\ \Ext^{n+1}(P, \syz X) \ar[r] & \Ext^{n+1}(Z, \syz X)\ar[r]^{\theta} & \Ext^{n+2}(\syz^{-1}Z, \syz X)\ar[r]& 0.}\]
Since $P,Q\in n$-$\proj\C$, $\Ext^{n+1}(P, \syz X)=0=\Ext^{n+1}(\syz^{-1}Z, Q)$, and so, we may deduce that $\im\alpha=\im\beta$.
\end{rem}

\begin{s}\label{use1}Assume that $M,N\in\C$ and $\ga\in\Ext^n(M,N)$ such that there is a morphism $f:M\rt P$ with $P\in n$-$\proj\C$ and a conflation $\et\in\Ext^n(P, N)$ such that $\ga=\et f$. Since $\C$ is $n$-Frobenius, there is an inflation $i:M\rt P'$, where $P'\in n$-$\proj\C$.
According to Lemma \ref{101}(1), $[f~~i]^t:M\rt P\oplus P'$ is also an inflation. Now by setting $\et':=\et\oplus(0\rt\cdots\rt 0\rt P'\rt P')$, we have
$\ga=\et'[f~~i]^t$. Consequently, without loss of generality, we may assume that $f$ is an inflation. Dually, if $\ga=g\be$, for some morphism
$g:Q\rt N$, with $Q\in n$-$\proj\C$ and $\be\in\Ext^n(M, Q)$, one may assume that $g$ is a deflation.
\end{s}

The result below is an immediate consequence of Remark \ref{use} and \ref{use1}. So  its proof is omitted.

\begin{prop}\label{equal}Let $X, Z$ be arbitrary objects of $\C$ and $\ga\in\Ext^n(Z, X)$. Then the following statements are equivalent:
\begin{enumerate}\item There is an object $\et\in\Ext^n(P, X)$, with $P\in n$-$\proj\C$, and a morphism $Z\st{f}\rt P$ such that $\ga=\et f$.
\item There is an object $\et'\in\Ext^n(Z, Q)$, with $Q\in n$-$\proj\C$, and a morphism $Q\st{f'}\rt X$ such that $\ga=f'\et'$.
\end{enumerate}
\end{prop}

\begin{s}\label{pp}
{\sc $\p$-subfunctor.} For every pair $X, Y$ of objects $\C$, let $\p(X, Y)$ denote the additive subgroup of $\Ext^n(X, Y)$ satisfying one of the equivalent conditions in Proposition \ref{equal}. It is easily seen that for given morphisms $A\st{f}\rt X$ and $Y\st{g}\rt B$ in $\C$, the natural transformation $\Ext(f, g):\Ext^n(X, Y)\lrt\Ext^n(A, B)$ respects $\p$. Namely, for any $\ga\in\p(X, Y)$, $g(\ga f)=(g\ga)f\in\p(A, B)$. Consequently, $\p$ is a subfunctor of $\Ext^n$, see \cite{as, fght}. Indeed, from our point of view, $\p$ is a submodule of $\Ext^n$. A given conflation $\ga\in\Ext^n(X, Y)$ will be called a {\em $\p$-conflation,} whenever $\ga$ belongs to $\p(X, Y)$. It is worth noting that in the case $n=0$, $\p$-conflations are those morphisms in $\C$ factoring through projective objects. If there is no ambiguity, we denote $\p(-, -)$ by $\p$.
\end{s}
The next result is also a direct consequence of Remark \ref{use} and \ref{use1}. So we ignore its proof.
\begin{cor}\label{lem2}Let $f:Z\rt P$ with $P\in n$-$\proj\C$, be an inflation. Then any $\p$-conflation $\ga\in\Ext^n(Z, X)$, factors through $f$.
\end{cor}

{

\begin{prop}\label{pprop}Let $M\st{f}\rt N\st{g}\rt K$ be a conflation in $\C$. Then, for any object $X\in\C$, there exists an exact sequence
$$\Ext^n(K, X)/{\p}\st{\bar{\g}}\lrt\Ext^n(N, X)/{\p}\st{\bar{\f}}\lrt\Ext^n(M, X)/{\p}.$$
\end{prop}
\begin{proof}
Take inflations $N\st{h}\rt P$ and $K\st{h'}\rt P'$, where $P, P'\in n$-$\proj\C$.
So, one may get the following commutative diagram: {\footnotesize
\[\xymatrix{M~\ar[r]^{hf} \ar[d]_{f}& P\ar[r]\ar[d]_{[1~~0]^t}& \syz^{-1}M\ar[d]_{f'}\\ N~\ar[r]^{[h~~h'g]^t} \ar[d]_{g} & P\oplus P'\ar[r] \ar[d]_{[0~~1]}& \syz^{-1}N \ar[d]_{g'} \\ K\ar[r]^{h'} & P' \ar[r] & \syz^{-1}K,}\]}whose rows and columns are conflation. Indeed, $hf$ is an inflation, because $f$ and $h$ are so. Moreover, $u:=[h~~h'g]^t$ is an inflation, thanks to Lemma \ref{101}. So, applying the functor $\Ext(-, X)$ to this diagram, gives us the following commutative diagram:
{\footnotesize
\[\xymatrix{\Ext^n(P', X)~\ar[r]^{\Ext^n(h', X)} \ar[d]& \Ext^n(K, X)\ar[r]^{\ga}\ar[d]_{\Ext^n(g, X)}& \Ext^{n+1}(\syz^{-1}K, X)\ar[d]\ar[r]&0\\ \Ext^n(P\oplus P', X)~\ar[r]^{\Ext^n(u, X)} \ar[d] & \Ext^n(N, X)\ar[r]^{\be} \ar[d]_{\Ext^n(f, X)}& \Ext^{n+1}(\syz^{-1}N, X) \ar[d]\ar[r]&0 \\ \Ext^n(P, X)\ar[r]^{\Ext^n(hf, X)} & \Ext^n(M, X) \ar[r]^{\et} & \Ext^{n+1}(\syz^{-1}M, X)\ar[r]&0,}\]}where rows and columns are exact. {Consequently, one may get the exact sequence, $\Ext^n(K, X)/{\im\Ext^n(h', X)}\rt\Ext^n(N, X)/{\im\Ext^n(u, X)}\rt\Ext^n(M, X)/{\im\Ext^n(hf, X)}$. In order to complete the proof, it suffices to show that $\p(K, X)=\im\Ext^n(h', X)$, because the two others can be obtained similarly. To this end, take an arbitrary object $\th\in\p(K, X)$. So, according to Corollary \ref{lem2}, $\th$ factors through $h'$, implying that $\th\in\im\Ext^n(h', X)$. Next take an object $\al\in\im\Ext^n(h', X)$. So there is an object $\al'\in\Ext^n(P', X)$ such that $\al' h'=\al$, and then, $\al$ is a $\p$-conflation, as desired. }
\end{proof}

\begin{s}\label{ccor}Let $M\rt P\rt\syz^{-1}M$ and $\syz N\rt Q\rt N$ be two arbitrary unit conflations in $\C$. According to the proof of Proposition \ref{pprop}, we may get the natural isomorphism $\Ext^{n+1}(\syz^{-1}M, N)\cong\Ext^n(M, N)/{\p}$. Also, a dual argument gives us the natural isomorphism $\Ext^{n+1}(M, \syz N)\cong\Ext^n(M, N)/{\p}$. These facts would imply the result below.
\end{s}
\begin{cor}\label{qo}A given morphism $f: M\rt N$ is quasi-invertible if and only if it acts as invertible on $\Ext^n/{\p}$.
\end{cor}

The next result indicates that being $\p$-conflation behaves well with respect to the pull-back and push-out along morphisms in $\si$.

\begin{prop}\label{nul} Let $a:N\rt X$ and $b:X'\rt X$ be two morphisms in $\si$. Then
\begin{enumerate}\item a given conflation $\ga\in\Ext^n(M, N)$ is a $\p$-conflation if and only if $a\ga$ is so.
\item a given object $\be\in\Ext^n(X, N)$ is a $\p$-conflation if and only if $\be b$ is so.
\end{enumerate}
\end{prop}
\begin{proof}By the similarity, we only prove the first assertion. Since the `only if' part follows from the fact the subfunctor $\p$ is closed under push-outs, we only prove the `if' part. To this end, assume that $a\ga$ is a $\p$-conflation. As $a\in\si$, Corollary \ref{qo}, the morphism $\Ext^n(M, N)/{\p}\lrt\Ext^n(M, X)/{\p}$ sending $\ga+\p$ to $a\ga+\p$ is an isomorphism. Now since $a\ga$ is $\p$-conflation, injectivity of this morphism yields that $\ga$ is a $\p$-conflation, as needed.
\end{proof}

\begin{cor}\label{div}Let $a:X\rt X'$ be a morphism in $\si$ and $Y\in\C$. Then
\begin{enumerate}\item for any $\ga\in\Ext^n(X, Y)$, there exists $\ga'\in\Ext^n(X', Y)$ such that $\ga-\ga'a$ is a $\p$-conflation. \item for a given $\be\in\Ext^n(Y, X')$, there exists $\be'\in\Ext^n(Y, X)$ such that $\be -a\be'$is a $\p$-conflation.
\end{enumerate}
\end{cor}
\begin{proof}Let us deal only with the first statement. The other one is obtained similarly. As $a:X\rt X'$ lies in $\si$, by Corollary \ref{qo}, the map $\Ext^n(X', Y)/{\p}\st{\bar{\a}}\lrt\Ext^n(X, Y)/{\p}$ is an isomorphism. So for a given object $\ga\in\Ext^n(X, Y)$, one may fined an object  $\ga'\in\Ext^n(X', Y)$ such that  $\p+\ga=\p+\ga' a$, meaning that $\ga-\ga'a$ is a $\p$-conflation, as required.
\end{proof}

\section{Unit factorizations of conflations}
In this section, we show that any conflation in $\C$ can be represented as a pull-back, as well as, a push-out of unit conflations. We begin with the following lemma.

\begin{lem}\label{gencog}Let $k\geq 1$ and $\ga\in\Ext^k(M, N)$. Then the following assertions hold:
\begin{enumerate}\item There exists an object $\ga'\in\Ext^k(M, N)$ and a {commutative diagram}

{\footnotesize \[\xymatrix{\ga:N~\ar[r]\ar@{=}[d]& X_{k-1}\ar[r]\ar[d]_{a_{k-1}}& \cdots \ar[r]& X_1\ar[r]\ar[d]_{a_1} &X_0\ar[r]\ar[d]_{a_0}& M\ar@{=}[d]\\ \ga':N~\ar[r] &Q_{k-1}\ar[r]&\cdots\ar[r]& Q_1\ar[r]& H\ar[r] & M,}\]}such that {rows are conflation,} ${Q_i}^{,}$s are $n$-projective and each $a_i$ is an inflation.

\item There exists an object $\ga''\in\Ext^k(M, N)$ and a {commutative diagram}
{\footnotesize \[\xymatrix{\ga'':N~\ar[r]\ar@{=}[d]& H'\ar[r]\ar[d]_{b_{k-1}}& P_{k-2}\ar[r]\ar[d]_{b_{k-2}}& \cdots \ar[r] &P_0\ar[r]\ar[d]_{b_0}& M\ar@{=}[d]\\ \ga:N~\ar[r] &{X_{k-1}}\ar[r]& {X_{k-2}}\ar[r]& \cdots \ar[r]& X_0\ar[r] & M,}\]}such that {whose rows are conflation, } all $P_i^,$s are $n$-projective and each $b_i$ is a deflation.
\end{enumerate}
\end{lem}
\begin{proof}Let us prove only the statement (1), since the other will be gained dually. To this end, we argue by induction on $k$. If $k=1$, then there is nothing to prove. So assume that $k\geq 2$ and the result has been proved for all integers smaller than $k$. Assume that $\ga=N\rt X_{k-1}\rt\cdots\rt X_0\rt M$. Set $\ga_k=N\rt X_{k-1}\rt L$ and $\ga^{k-1}=L\rt X_{k-2}\rt\cdots\rt X_0\rt M$, to get the equality $\ga=\ga_k\ga^{k-1}$. Take the following commutative diagram:
\[\xymatrix{\ga_k:N\ar@{=}[d]\ar[r]&X_{k-1}\ar[r]\ar[d]_{a_{k-1}}&L\ar[d]_{a} \\ \delta':N\ar[r]& Q_{k-1} \ar[r]& T,}\]
where $a_{k-1}$ is an inflation with $Q_{k-1}\in n$-$\proj\C$. Since $a\ga^{k-1}\in\Ext^{k-1}(M, T)$, by the induction hypothesis, there is a conflation $\ga_1:T\rt Q_{k-2}\rt\cdots\rt Q_1\rt H\rt M$ and a {commutative diagram  {\footnotesize \[\xymatrix{a\ga^{k-1}:T~\ar[r]\ar@{=}[d]& Y\ar[r]\ar[d]& X_{k-3}\ar[r]\ar[d]_{a_{k-3}}& \cdots \ar[r]& X_1\ar[r]\ar[d]_{a_1} &X_0\ar[r]\ar[d]_{a_0}& M\ar@{=}[d]\\ \ga_1:T~\ar[r] &Q_{k-2}\ar[r]& Q_{k-3}\ar[r]&\cdots\ar[r]& Q_1\ar[r]& H\ar[r] & M.}\]}}
 Hence, by setting $\ga':=\delta'\ga_1$ and using the fact that $\ga=(\delta' a)\ga^{k-1}$, {one obtains the desired result.} 
\end{proof}

\begin{prop}\label{102}Let
$M, N\in\C$ and let $\ga\in\Ext^k(M, \syz^kN)$ with $k\geq 1$. Then the following assertions hold:
\begin{enumerate}
\item There exists a unit conflation $\delta\in\U^k(\syz^kN)$ and $f\in\H$ such that $\ga=\delta f$.
\item There exists a unit conflation $\delta \in\U_k(M)$ and $g\in\H$ such that $\ga=g\delta$.
\end{enumerate}
\end{prop}
\begin{proof}We only prove the first assertion. The second one is obtained dually. By Lemma \ref{gencog} together with \cite[Chapter VII, Proposition 3.1]{mit}, we may assume that $\ga$ has the form $\syz^kN\rt P_{k-1}\rt\cdots\rt P_1\rt H\rt M$, where $P_i\in n$-$\proj\C$, for any $i$. Taking the conflation $L\rt H\rt M$ and an inflation $H\rt P_0$ with $P_0\in n$-$\proj\C$, one gets the following commutative diagram: \[\xymatrix{L~\ar[r]\ar@{=}[d]& H\ar[r]\ar[d]&M \ar[d]_{f}\\ L~\ar[r] & P_0\ar[r]& N'.}\]
Considering $\delta:=\syz^kN\rt P_{k-1}\rt\cdots\rt P_1\rt P_0\rt N'$, one has the equality $\ga=\delta f$, as desired.
\end{proof}

\begin{dfn}Assume that $\ga\in\Ext^k(M, \syz^kN)$ with $k\geq 1$, is given. Assume that there is a unit conflation $\delta\in\U^k(\syz^kN)$ and $f\in\H$ such that $\ga=\delta f$. Then we say $\ga$ factors through $f$ by the unit conflation $\delta$, and the equality $\ga=\delta f$ is said to be {\it a right unit factorization} (abb. $\ruf$) of $\ga$. Dually, if there exists $\delta\in\U_k(M)$ and $g\in\H$ such that $\ga=g\delta$, then we call this a {\it left unit factorization} (abb. $\luf$) of $\ga$.\\ In view of Proposition \ref{102}, every conflation admits an $\ruf$, as well as, an $\luf$.
\end{dfn}

}

\begin{dfn}Assume that $a$ is a morphism in $\si$ and $k\geq 1$. We say that $a$ is
\begin{enumerate}\item {\it induced by identity over $N$}, provided that there are unit conflations $\delta, \delta'\in\U_k(N)$ such that $\delta=a\delta'$ and $a$ is inflation.\item {\it co-induced by identity over $N$}, if there exist $\delta, \delta'\in\U^k(N)$ such that $\delta=\delta'a$ and $a$ is deflation.
\end{enumerate}
\end{dfn}

\begin{prop}\label{pro100}Let $k\geq 1$. Then the following statements hold:
\begin{enumerate}\item For given $\delta_1, \delta_2,\delta_3\in\U^k(N)$, there exist deflations {$a_1,a_2,a_3\in\si$} such that $\delta_1 a_1=\delta_2a_2=\delta_3a_3$. In particular, $a_1, a_2, a_3$ are co-induced by identity over $N$.
\item For given $\delta_1, \delta_2, \delta_3 \in\U_k(N)$, there exist inflations {$a_1,a_2,a_3\in\si$} such that $a_1\delta_1=a_2\delta_2=a_3\delta_3$. In particular, $a_1, a_2,a_3$ are induced by identity over $N$.
\end{enumerate}
\end{prop}
\begin{proof}We deal only with the first assertion, the second one is obtained dually.
Assume that $\delta_i=N\st{j_i^0}\lrt P_i^1\lrt P_i^2\lrt\cdots\lrt P_i^k\lrt N_i$, for any $1\leq i\leq 3$. According to the proof of Lemma \ref{101}, one gets the conflation $N\st{j^0}\lrt\oplus_{i=1}^{3}P_i^1\lrt L^1$, where $j^0=[j_1^0~~j_2^0~~j_3^0]^t$. So we may have the following commutative diagram: \[\xymatrix{N~\ar[r]\ar@{=}[d]&\oplus_{i=1}^{3} P_i^1\ar[r]\ar[d]_{e_i^1}& L^1\ar[d]_{b_i^1}\\ N~\ar[r] & P_i^1\ar[r]& L_i^1,}\] where $e_i^1$ is the projection, {and then, the snake lemma yields that $b_i^1$ is a deflation with  $n$-projective kernel.} Now by taking an inflation $L^1\st{u^1}\lrt Q^2$, we obtain a conflation $L^1\st{j^1}\lrt\oplus_{i=1}^3P_i^2\oplus Q^2\lrt L^2$, where $j^1=[j_1^1b_1^1~~j_2^1b_2^1~~j_3^1b_3^1~~u^1]^t$ and $j_i^1:L_i^1\rt P_i^2$ is inclusion, for any $1\leq i\leq 3$. Consequently, there exists a commutative diagram \[\xymatrix{L^1~\ar[r]\ar[d]_{b_i^1}&\oplus_{i=1}^{3} P_i^2\oplus Q^2\ar[r]\ar[d]_{e_i^2}& L^2\ar[d]_{b_i^2}\\ L_i^1~\ar[r] & P_i^2\ar[r]& L_i^2,}\] where $e_i^2$ is the projection. {Another use of the sanke lemma, together with the fact that $n$-$\proj\C$ is closed under cokernels of monomorphisms, forces $b_i^2$ to be a deflation with  $n$-projective kernel, see \cite[Corollary 3.6]{buh}.} Continuing this procedure and splicing the diagrams, gives us the following commutative diagram:{\footnotesize \[\xymatrix{ \delta_4:N~\ar[r]^{j^0}\ar@{=}[d] & \oplus_{i=1}^3P_i^1\ar[d]_{e_i^1}\ar[r] & \oplus_{i=1}^3P_i^2\oplus Q^2\ar[d]_{e_i^2}\ar[r]&\cdots \ar[r]& \oplus_{i=1}^3P_i^k\oplus Q^k\ar[d]_{e_i^k}\ar[r]\ar[d]_{e_i^k} & N_4\ar[d]_{a_i}\\ \delta_i:N~\ar[r]^{j_i^0} & P_i^1\ar[r]& P_i^2\ar[r]& \cdots \ar[r] & P_i^k\ar[r] & N_i,}\]}where each $e^j_i$ is the projection  and  $a_i$ is a deflation with $n$-projective kernel, and so, it belongs to $\si$, thanks to Corollary \ref{is}. In particular,  Remark \ref{pp1}(1) yields that $\delta_4=\delta_i a_i$, for any $1\leq i\leq 3$. So the proof is finished.
\end{proof}

\begin{dfn}(1) Assume that $\delta, \delta' \in \U^k(N)$ with $k\geq 1$ and $a, a'$ are co-induced morphisms by identity over $N$ such that $\delta a=\delta'a'$. Then we say that $[\delta a, \delta'a']$
is a {\it co-angled pair}. In some cases, based on our need, we denote it by $\delta\st{a}\llf\delta''\st{a'}\lrt\delta'$, where $\delta''=\delta a$. One should note that by Proposition \ref{pro100} such a co-angled pair exists.\\
(2) Assume that $\delta_1, \delta_2 \in\U_k(N)$ with $k\geq 1$ and $a_1,a_2$ are induced morphisms by identity over $N$. Then we say that $[a_1\delta_1, a_2\delta_2]$ is {\it an angled pair}, if $a_1\delta_1=a_2\delta_2$. Sometimes we display it by $\delta_1\st{a_1}\lrt\delta_3\st{a_2}\llf\delta_2$, whenever $\delta_3=a_1\delta_1$.
\end{dfn}

\begin{prop}\label{coin}Let $\ga\in\Ext^k(M, \syz^k N)$ with $k\geq 1$ and let $\ga=\delta f=\delta'f'$ be two $\ruf$s of $\ga$. Let $\delta\st{a}\llf\delta''\st{a'}\lrt\delta'$ be a co-angled pair which is obtained in Proposition \ref{pro100}.
Then there is a morphism $h\in\C$ such that $f=ah$ and $f'=a'h$. Particularly, $\ga=\delta''h$ is an $\ruf$ of $\ga$.
\end{prop}
\begin{proof}
By the hypothesis, there are commutative diagrams{\footnotesize \[\xymatrix{\ga:\syz^kN~\ar[r]\ar@{=}[d]& X_{k-1}\ar[r]\ar[d]_{f_{k-1}}& \cdots \ar[r]& X_0\ar[r]\ar[d]_{f_0}& M\ar[d]_{f}\\ \delta:\syz^kN~\ar[r] &P_{k-1}\ar[r]&\cdots\ar[r]& P_0\ar[r]& N,}\] \[\xymatrix{\ga:\syz^kN~\ar[r]\ar@{=}[d]& X_{k-1}\ar[r]\ar[d]_{f'_{k-1}}& \cdots \ar[r]& X_0\ar[r]\ar[d]_{f'_0}& M\ar[d]_{f'}\\ \delta':\syz^kN~\ar[r] &P'_{k-1}\ar[r]&\cdots\ar[r]& P'_0\ar[r]& N'.}\]}So according to the proof of Proposition \ref{pro100}, one may obtain the following commutative diagram: {\footnotesize \[\xymatrix{\ga:\syz^kN~\ar[r]\ar@{=}[d]& X_{k-1}\ar[r]\ar[d]_{{{\tiny {\left[\begin{array}{ll} f_{k-1} \\ f'_{k-1} \end{array} \right]}}}}& \cdots \ar[r]& X_0\ar[r]\ar[d]_{{{{\tiny {\left[\begin{array}{ll} f_0 \\ f'_0 \\ 0 \end{array} \right]}}}}} & M\ar[d]_{h}\\ \delta'':\syz^kN~\ar[r] &P_{k-1}\oplus P'_{k-1}\ar[r]&\cdots\ar[r]& P_0\oplus P'_0\oplus Q_0\ar[r]& N'',}\]}namely $\delta''h=\ga$. Now since $\delta''=\delta a=\delta'a'$, we have $\delta(ah)=\delta'(a'h)=\ga$. Particularly, one has the commutative diagram
{\footnotesize\[\xymatrix{N~& N''\ar[l]_{a}\ar[r]^{a'}& N'\\ & M\ar[u]^{h}\ar[ul]^{f}\ar[ur]_{f'} ,& }\]}
meaning that $f=ah$ and $f'=a'h$. Thus the proof is complete.
\end{proof}

The result below can be obtained by dualizing the argument given in the proof of Proposition \ref{coin}. So its proof will be omitted.
\begin{prop}\label{lif}Let $k\geq 1$ and $\ga\in\Ext^k(M, \syz^k N)$ and let $\ga=f\delta=f'\delta'$ be two $\luf$s of $\ga$. Assume that $\delta\st{a}\lrt\delta''\st{a'}\llf\delta'$ is an angled pair which is obtained in Proposition \ref{pro100}. Then there exists $h\in\H$ such that $\ga=h\delta''$ is also an $\luf$ of $\ga$. In particular, $f=ha$ and $f'=ha'$.
\end{prop}

\section{$n$-$\Ext$-phantom morphisms}

Assume that $\C$ is an $n$-Frobenius category. In this section, we will see that an object $f\in\H$ annihilates $\Ext^n/{\p}$ from the left and the right, simultaneously. We call such a morphism $f$, an {\em $n$-$\Ext$-phantom morphism}. This notion has a connection with many branches of mathematics, see \ref{s100}. We begin with the following easy observation.

\begin{s}\label{cof}Assume that $f:X\rt N$ is a morphism in $\C$ factoring through an $n$-projective object. Then $f$ annihilates $\Ext^n/{\p}$ from the both sides. To see this, take $P\in n$-$\proj\C$ and morphisms $f_1, f_2$ such that $f:X\st{f_1}\lrt P\st{f_2}\lrt N$. So for a given object $\ga\in\Ext^n(N, Y)$, $\ga f=(\ga f_2)f_1$ will be a $\p$-conflation, because $\ga f_2\in\Ext^n(P, Y)$. Namely, $(\Ext^n/{\p})f=0$. The same method reveals that $f(\Ext^n/{\p})=0$.
\end{s}

The next interesting result says that a given morphism $f$ annihilates $\Ext^n/{\p}$, whenever it annihilates some unit conflation $\delta$.
\begin{prop}\label{three}
Let $n\geq 1$ and $f:X\rt N$ be a morphism in $\C$. Then the following are equivalent: \begin{enumerate} \item There exists a unit conflation $\delta\in\U_n(N)$ such that $\delta f$ is a $\p$-conflation.\item
For any unit conflation $\delta'\in\U_n(N)$, $\delta' f$ is a $\p$-conflation.\item $(\Ext^n/{\p})f=0$. \end{enumerate}
\end{prop}
\begin{proof}
$(1)\Rightarrow (2)$ Take an arbitrary unit conflation $\delta'\in\U_n(N)$. We intend to show that $\delta'f$ is a $\p$-conflation. To see this, consider an angled pair
$\delta\st{a}\rt\delta''\st{b}\lf \delta'$. Since $\delta f$ is a $\p$-conflation,   Proposition \ref{nul}(1)  implies that the same is true for $a(\delta f)=(a\delta)f$.
Moreover, as $a\delta=b\delta'$, another use of Proposition \ref{nul}(1) enables us to infer that $\delta'f$ is also a $\p$-conflation, as needed.\\
$(2)\Rightarrow (3)$ Assume that $\ga\in\Ext^n(N, M)$ is an arbitrary conflation and $\ga=g\delta'$ is an $\luf$ of $\ga$, where $\delta'\in\U_n(N)$. By the hypothesis, $\delta'f$ is a $\p$-conflation. So using the fact that being a $\p$-conflation is preserved under push-out, we infer that $\ga f=g(\delta'f)$ is also a $\p$-conflation, that is, $\ga f\in\p$.
Consequently, $(\Ext^n/{\p})f=0$ .\\
$(3)\Rightarrow (1)$ This implication is obvious. So the proof is finished.
\end{proof}

As an immediate consequence of Proposition \ref{three}, we include the next result.
\begin{cor}\label{ccoo}Let $n\geq 1$ and ${\ga}\in\Ext^n(M, \syz^n
N)$ be a conflation and let $\ga=\delta f$ be an $\ruf$
of $\ga$. Then $\ga$ is a $\p$-conflation if and only if $(\Ext^n/{\p})f=0$.
\end{cor}

\begin{lem}\label{ruf}Let $\ga\in\Ext^n(M, \syz^nN)$ be a $\p$-conflation. Then there is an $\ruf$ $\ga=\delta f$ of $\ga$ with $f$ factoring through an $n$-projective object.
\end{lem}
\begin{proof}Since $ \ga $ is a $ \p $-conflation, there exist an object $\et\in \Ext^n( P , \syz^n N)$ with $P\in n$-$\proj\C$, and a morphism $ h: M\rt P $ such that $\ga= \et h$. Taking an $\ruf$ $\et=\delta g$ of $\et$, we get $\ga=\delta(gh)$ is an $\ruf$ of $\ga$, in which $gh$ factors through the $n$-projective object $P$. So we are done.
\end{proof}

\begin{prop}\label{ph}Let $f:M\rt N$ be a morphism in $\C$ and $n\geq 1$. If $(\Ext^n/{\p})f=0$, then there are morphisms $N\st{a}\lf N''\st{b}\rt N'$ with $a, b\in\si$ and $h:M\rt N''$ such that $f=ah$ and $bh$ factors through an $n$-projective object.
\end{prop}
\begin{proof}Fix a unit conflation $\delta_N\in\Ext^n(N, \syz^nN)$. By our hypothesis, $\delta_Nf$ is a $\p$-conflation. By Lemma \ref{ruf}, there is an $\ruf$ $\delta_Nf=\delta g$ of $\delta_Nf$ such that $g$ factors through an $n$-projective object. Consider a co-angled pair $\delta_N\st{a}\leftarrow \delta''\st{b}\rt \delta$, as the one obtained in Proposition \ref{pro100}. Indeed, we have a pair of morphisms $N\st{a}\lf N''\st{b}\rt N'$, where $N'$ (resp. $N''$) is the right end term of the unit conflation $\delta$ (resp. $\delta''$). So applying Proposition \ref{coin} ensures the existence of a morphism $h: M\rt N''$ such that $\delta_Nf=\delta'' h$ is also an $\ruf$ of $\delta_Nf$. Particularly, $f=ah$ and $bh$ factors through an $n$-projective object, because of the equality $bh=g$. So the proof is completed.
\end{proof}

\begin{prop}\label{sif}Let $a:X\rt N$ be a morphism in $\si$. Then the following are satisfied:
\begin{enumerate}\item A given morphism $f:N\rt M$ annihilates $\Ext^n/{\p}$ from the right if and only if so does $fa$. \item A given morphism $g:M\rt X$ annihilates $\Ext^n/{\p}$ from the left if and only if so does $ag$.
\end{enumerate}
\end{prop}
\begin{proof}By the similarity, we only prove the first statement.
If $(\Ext^n/{\p})f=0$ , evidently the same is true for$(\Ext^n/{\p})(fa)$, because being $\p$-conflation is stable under pull-back. Now assume that $(\Ext^n/{\p})(fa)=0$. So, by Proposition \ref{three}, there exists a unit conflation $\delta_M\in\U_k(M)$ such that $\delta_M(fa)$ is a $\p$-conflation. Hence, invoking Proposition \ref{nul}(2) yields that $\delta_Mf$ is a $\p$-conflation, as well. Consequently, another use of Proposition \ref{three} forces $(\Ext^n/{\p})f$ to be zero, as required.
\end{proof}
Now we are ready to state the main result of this section.
\begin{theorem}\label{main}
Let $f:M\rt N$ be a morphism in $\C$. Then $(\Ext^n/{\p})f=0$ if and only if $f(\Ext^n/{\p})=0$.
\end{theorem}
\begin{proof}Since the result in the case $n=0$ holds obviously, we may assume that
$n\geq 1$. Suppose that $(\Ext^n/{\p})f=0$. In view of Proposition \ref{ph}, there are morphisms $N\st{a}\lf N''\st{b}\rt N'$ with $a, b{\in\si}$ and $h:M\rt N''$ such that $f=ah$ and $bh$ factors through an $n$-projective object. So, as we have observed in \ref{cof}, $(bh)(\Ext^n/{\p})=0$. As $a, b\in\si$, applying Proposition \ref{sif} successively, yields that $f(\Ext^n/{\p})=0$. Since the reverse implication is obtained analogously, we ignore it. So, the proof is finished.
\end{proof}

\begin{dfn} A morphism $f$ in $\C$ is called an {\em $n$-$\Ext$-phantom morphism}, if it annihilates $\Ext^n/{\p}$.
\end{dfn}

The result below follows directly from \ref{ccor}. So we skip its proof.
\begin{cor}A given morphism $f$ in $\C$ is an $n$-$\Ext$-phantom morphism if and only if $\Ext^{n+1}f=0$.
\end{cor}

\begin{rem}It should be noted that our notion of $n$-$\Ext$-phantom morphisms, is indeed $(n+1)$-$\Ext$-phantom morphisms in the sense of Mao \cite{mao}. However, due to the harmony with the $n$-Frobenius category, we call them $n$-$\Ext$-phantom morphisms. It should be also emphasized that  a 1-$\Ext$-phantom morphism is exactly a $\p$-phantom morphism in the sense of Fu et. al. \cite{fght}.
\end{rem}

\begin{s}\label{s100}The concept of phantom morphisms has its roots in topology in the study of maps between CW-complexes \cite{mc}. A map $f: X \rt Y$ between CW-complexes is said to be a phantom map, if its restriction to each skeleton $X_n$ is null homotopic. Later, this notion has been used in various settings of mathematics. In the context of triangulated categories, phantom morphisms were first studied by Neeman \cite{ne}. The notion of phantom morphisms also was developed in the stable category of a finite group ring in a series of works by Benson and Gnacadja \cite{gn, be2, be1, be}. The definition of a phantom morphism was generalized to the category of $R$-modules over an associative ring $R$ by Herzog \cite{he}. Precisely, a morphism $f: M\rt N$ of left $R$-modules is called a phantom morphism, if the natural transformation $\Tor^R_1(-, f):\Tor^R_1(-, M)\rt\Tor^R_1(-, N)$ is zero, or equivalently, the pullback of any short exact sequence along $f$ is pure exact. Similarly, a morphism $g:M\rt N$ of left $R$-modules is said to be an $\Ext$-phantom morphism \cite{hext}, if the induced morphism $\Ext^1_R(B, g):\Ext^1_R(B,M)\rt\Ext^1_R(B,N)$ is 0, for every finitely presented left $R$-module $B$. For any integer $n\geq 1$, the concepts of $n$-phantom morphism and $n$-$\Ext$-phantom morphisms, which are higher dimensional generalizations of phantom morphisms and $\Ext$-phantom morphisms, respectively, have been introduced and studied by Mao \cite{mao, mao1}.
\end{s}

\section{A composition operator on $\Ext^n/{\p}$}

Assume that $\C$ is an $n$-Frobenius category. In this section, we introduce a composition operator $``\circ"$ on $\Ext^n/{\p}$. It is proved that the operator $``\circ"$ is associative and distributive over Baer sum on both sides. These facts enable us to see that for any object $M\in\C$, $(\Ext^n(M, \syz^nM)/{\p}, +, \circ)$ is a ring with identity, where $+$ stands for the Baer sum operation. Surprisingly, we find a ring homomorphism $\varphi:\hom_{\C}(M, M)\lrt\Ext^n(M, \syz^nM)/{\p}$ such that for any quasi-invertible morphism $f$, $\varphi(f)$ is invertible and $ \varphi(f)=0$, if $f$ is an $n$-$\Ext$-phantom morphism. We begin with the following notation.

\begin{s}Let $a:X\rt X'$ be a quasi-invertible morphism and $\ga\in\Ext^n(X, Y)$. In view of Corollary \ref{div}, there exists $\ga'\in\Ext^n(X', Y)$ such that $\ga-\ga'a$ is a $\p$-conflation. In this case, for simplicity, we write $\ga'=_{_{\p}}\ga a^{-1}$. Also, {for a given object} $\be\in\Ext^n(Y, X')$, there exists $\be'\in\Ext^n(Y, X)$ such that $\be -a\be'$ is a $\p$-conflation. Then we write $\be'=_{_{\p}}a^{-1}\be$. {From now on, a given object $\ga+\p\in\Ext^n/{\p}$, will be denoted by $\bar{\ga}$.}
\end{s}

\begin{dfn}\label{compo}Assume that $M, N, K\in\C$ and fix a unit conflation $\delta_N\in\Ext^n(N, \syz^nN)$. We define the composition
$$\circ: \Ext^n(N, \syz^nK)/{\p}\times\Ext^n(M, \syz^nN)/{\p}\lrt\Ext^n(M, \syz^nK)/{\p},$$ $(\bar{\be}, \bar{\ga})\lrt\bar{\be}\circ\bar{\ga}$ as follows: \\
Assume that $\ga=\delta f$ is an $\ruf$ of $\ga$, where $\delta\in\U^n(\syz^nN)$ and $f\in\H$.
Now we set $\bar{\be}\circ\bar{\ga}:=\overline{((\be a_1)b^{-1}_1)f}$, where $\delta_N\st{a_1}\llf\delta_1\st{b_1}\lrt\delta$
is a co-angled pair.
\end{dfn}

The result below allows us to assume that the co-angled pair in the above definition, as the one obtained in Proposition \ref{pro100}(1).
\begin{lem}\label{ds}With the notation above, assume that $\delta_N\st{a}\llf\delta''\st{b}\lrt\delta$ is the co-angled pair which is obtained in Proposition \ref{pro100}(1). Then $((\be a_1)b^{-1}_1)f=_{\p}((\be a)b^{-1})f$.
\end{lem}
\begin{proof} Since $\delta_N\st{a_1}\llf\delta_1\st{b_1}\lrt\delta$ is a co-angled pair, we have $\delta_1=\delta_Na_1=\delta b_1$. So by Proposition \ref{coin}, $\delta_1=\delta''h$, for some morphism $h$ in $\C$. Indeed, the argument given in the proof of Proposition \ref{coin}, gives rise to the following commutative diagram:
{\footnotesize\[\xymatrix{N~& N''\ar[l]_{a}\ar[r]^{b}& N'\\ & N_1\ar[u]^{h}\ar[ul]^{a_1}\ar[ur]_{b_1} ,& }\]}where $N', N''$ and $N_1$ stand for the right end terms of $\delta, \delta''$ and $\delta_1$, respectively. This, in particular, implies that {$((\be a_1)b^{-1}_1)f=_{\p}((\be ah)(bh)^{-1})f=_{\p}((((\be a)h)h^{-1})b^{-1})f=_{\p}((\be a)b^{-1})f$}, giving the desired result.
\end{proof}

\begin{theorem}\label{welldef}The definition of $``\circ"$ is independent of the choice of an $\ruf$ of $\ga$.
\end{theorem}
\begin{proof}
Assume that $\ga=\delta'f'$ is another $\ruf$ of $\ga$ and suppose that $\delta_N\st{a_2}\llf\delta_2\st{b_2}\lrt\delta'$
is a co-angled pair. We shall prove that $((\be a_1)b_1^{-1})f=_{\p}((\be a_2)b_2^{-1})f'$. Take a co-angled pair
$\delta\st{a_3}\llf\delta_3\st{b_3}\lrt\delta'$. According to Proposition \ref{pro100}(1), there exist $a_4,b_4,c_4\in\si$ such that $\delta_1a_4=\delta_2b_4=\delta_3c_4$ and $a_4,b_4,c_4$ are co-induced by identity over $\syz^nN$. In particular, denoting the latter equalities by $\delta_4$, we will get the following diagram of co-angled pairs: {\footnotesize \[\xymatrix{&\delta_N& \\ \delta_1~\ar[d]_{b_1}\ar[ur]^{a_1}& \delta_4\ar[l]_{a_4}\ar[r]^{b_4}\ar[d]_{c_4}& \delta_2\ar[d]_{b_2}\ar[ul]_{a_2}\\ \delta~ & \delta_3\ar[l]_{a_3}\ar[r]^{b_3} & \delta'.}\]}{Since by Lemma \ref{ds}, the co-angled pair $\delta\st{a_3}\llf\delta_3\st{b_3}\lrt\delta'$ can be considered as the one obtained in Proposition \ref{pro100}(1), }
by virtue of Proposition \ref{coin}, there is a morphism $h$ such that $\ga=\delta_3h$, $f=a_3h$ and $f'=b_3h$. In particular, one may have the following diagram: {\footnotesize\[\xymatrix{\delta~& \delta_3\ar[l]_{a_3}\ar[r]^{b_3}& \delta'\\ & \ga\ar[u]^{h}\ar[ul]^{f}\ar[ur]_{f'} .& }\]}
Thus we have the following equalities modulo $\p$: $$((\be a_2)b_2^{-1})f'=_{\p}((\be a_2)b_2^{-1})b_3h=_{\p}(((\be a_2)b_2^{-1})b_3)h=_{\p}((((\be a_2)b_2^{-1})b_3)a_3^{-1})f$$ $$=_{\p}((((\be a_2)b_4){c_4}^{-1})a_3^{-1})f=_{\p} (((\be(a_2b_4){a_4}^{-1})b_1^{-1})f=_{\p}((\be a_1)b_1^{-1})f.$$
Here the first and third equalities hold, because the latter diagram is commutative, and the second one is clear.
The validity of the fourth and fifth equalities comes from the fact that $b_3c_4=b_2b_4$ and $b_1a_4=a_3c_4$, respectively. Finally, the last one holds true, because $a_2b_4=a_1a_4$. Thus
$((\be a_1)b_1^{-1})f=_{\p}((\be a_2)b_2^{-1})f'$, as needed.
\end{proof}

Let $M, N$ be two objects of $\C$. For given two objects $\bar{\ga}, \bar{\ga'}\in\Ext^n(M, \syz^nN)/{\p
}$, we define the addition $\bar{\ga}+\bar{\ga'}:=\overline{\ga+\ga'}$, where $\ga+\ga'$ is the usual Baer sum operation. Evidently, this definition is well-defined.

\begin{prop}\label{srt1}The operator $``\circ"$ is distributive over $``+"$ on both sides.
\end{prop}
\begin{proof} Assume that $ K, M,N\in \C $ and fix unit conflations $\delta_N\in\Ext^n(N, \syz^nN)$ and $\delta_K\in\Ext^n(K, \syz^nK)$. Assume that $\bar{\al}\in\Ext^n(M, \syz^n N)/{\p}$, $\bar{\ga}\in \Ext^n(K, \syz^n L)/{\p}$ and $\bar{\be}, \bar{\be'}\in \Ext^n(N, \syz^n K)/{\p}$. First, we show that
$(\bar{\be}+\bar{\be'})\circ\bar{\al}=\bar{\be}\circ\bar{\al}+ \bar{\be'}\circ \bar{\al}$. In this direction, assume that
$\al=\delta f$ is an $\ruf$ of $\al$ and suppose that $[\delta_Na, \delta b]$ is a co-angled pair. So, $(\bar{\be}+\bar{\be'})\circ\bar{\al}=\overline{(((\be+\be')a)b^{-1})f}=\overline{((\be a)b^{-1})f}+\overline{((\be' a)b^{-1})f}=\bar{\be}\circ\bar{\al}+ \bar{\be'}\circ \bar{\al}$, where the second equality follows from the fact that pull-back distributes over the Baer sum, see \cite[Chapter VII, Lemma 3.2]{mit}. Next we would like to show that $\bar{\ga }\circ(\bar{\be}+\bar{ \be'})=\bar{\ga}\circ\bar{\be}+\bar{\ga}\circ\bar{\be'}$. Assume that $\be=\delta f$ and $\be'=\delta' f'$ are $\ruf$s of $\be$ and $\be'$, respectively. Take an $\ruf$, $\nabla(\delta\oplus\delta')=\delta''g$, where $\nabla:\syz^nK\oplus\syz^nK\st{[1~~1]}\lrt\syz^nK$.
Namely, there is a commutative diagram \[\xymatrix{\nabla(\delta\oplus\delta'):\syz^n K~\ar[r]\ar@{=}[d]& T_{n-1}\ar[r]\ar[d] &\cdots\ar[r] &{P_0\oplus P'_0}\ar[r]\ar[d] & K'\oplus K''\ar[d]_{g}\\ \delta'':\syz^nK ~\ar[r] & {Q}_{n-1}\ar[r]& \cdots\ar[r]& Q_0\ar[r]& K_1,}\] where each $Q_i$ is $n$-projective. Now setting $g:=[g'~~g'']$,
one may easily deduce that $\be=\delta''(g'f)$ and $\be'=\delta''(g''f')$ are also $\ruf$s of $\be$ and $\be'$, respectively. We claim that $\be+\be'=\delta''h$, for some morphism $h$ in $\C$. Since there is a { commutative diagram\[\xymatrix{\be\oplus\be':\syz^n K\oplus\syz^nK~\ar[r]\ar@{=}[d]& X_{n-1}\oplus Y_{n-1}\ar[r]\ar[d] &\cdots\ar[r] &X_0\oplus Y_0\ar[r]\ar[d] & N\oplus N\ar[d]_{f\oplus f'}\\ \nabla(\delta\oplus\delta'):\syz^nK ~\ar[r] & {T}_{n-1}\ar[r]& \cdots\ar[r]& P_0\oplus P'_0\ar[r]& K'\oplus K'',}\]
the universal property of the push-out yields the existence of the following commutative diagram:

\[\xymatrix{\nabla(\be\oplus\be'):\syz^n K~\ar[r]\ar@{=}[d]& L_{n-1}\ar[r]\ar[d] &\cdots\ar[r] &X_0\oplus Y_0\ar[r]\ar[d] & N\oplus N\ar[d]_{f\oplus f'}\\ \nabla(\delta\oplus\delta'):\syz^nK ~\ar[r] & {T}_{n-1}\ar[r]& \cdots\ar[r]& P_0\oplus P'_0\ar[r]& K'\oplus K''.}\]} Indeed, we have $\nabla(\be\oplus\be')=\nabla(\delta\oplus\delta')(f\oplus f')$, {see Remark \ref{pp1}(1)}. As $\be+\be'=\nabla(\be\oplus\be')\Delta$, where $\Delta:N\st{{{\tiny {\left[\begin{array}{ll} 1 \\ 1 \end{array} \right]}}}}\lrt N\oplus N$, if $h:=g(f\oplus f')\Delta$, we have $\be+\be'=\delta''h$, as claimed. Therefore, considering the co-angled pair $[\delta_Ka_1, \delta''b_1]$, we have $\bar{\ga}\circ(\bar{\be}+\bar{ \be'})=\overline{((\ga a_1)b_1^{-1})h}$. Now since $h=[g'~~g''](f\oplus f')\Delta=g'f+g''f'$, we infer that $\bar{\ga}\circ(\bar{\be}+\bar{ \be'})=\overline{((\ga a_1)b_1^{-1})g'f}+\overline{((\ga a_1)b_1^{-1})g''f'}=\bar{\ga}\circ\bar{\be}+\bar{\ga}\circ\bar{\be'}$. So the proof is complete.
\end{proof}

\begin{lem}\label{corwell}
Let ${\ga}\in\Ext^n(M, \syz^nN)$ and ${\be}\in\Ext^n(N, \syz^nK)$. If $\ga$ or $\be$ is a $\p$-conflation, then $\bar{\be}\circ\bar{\ga}=0$. In particular, if $\be'$ and $\ga'$ are two conflations such that $\bar{\be}=\bar{\be'}$ and $\bar{\ga}=\bar{\ga'}$, then $\bar{\be}\circ\bar{\ga}=\bar{\be'}\circ\bar{\ga'}$.
\end{lem}
\begin{proof} First assume that $ \be$ is a $ \p $-conflation. Let $ \ga= \delta f $ be an $\ruf$ of $ \ga $ and let  $[\delta_N a , \delta b]$ be a co-angled pair, where $\delta_N\in\Ext^n(N, \syz^nN)$ is a unit conflation.
As $\be$ is a $\p$-conflation, Proposition \ref{nul} ensures that $(\be a)b^{-1}$ is a $\p$-conflation, and so, the same will be true for $((\be a)b^{-1})f$, meaning that $\bar{\be}\circ\bar{\ga}=\bar{0}$.\\
Next assume that $ \ga $ is a $ \p $-conflation. In view of Lemma \ref{ruf}, there is an $\ruf$ $\ga=\delta f$ of $\ga$ such that $f$ factors through an $n$-projective object.
Consider a co-angled pair $[\delta_N a, \delta b]$. According to the proof of Theorem \ref{welldef}, $\bar{\be}\circ\bar{\ga}=\overline{((\be a) b^{-1})f}$.
Now since $f$ factors through an $n$-projective object, it is evident that $((\be a) b^{-1})f$ is a $\p$-conflation, and then, $\bar{\be}\circ\bar{\ga}=\bar{0}$. The second assertion follows by combining the first assertion with Proposition \ref{srt1}. Thus the proof is finished
\end{proof}

\begin{rem}\label{bimod}Let $X, Y$ be two arbitrary objects in $\C$ and $\be, \be'\in\Ext^n(X, Y)$. Assume that there is a {commutative diagram
\[\xymatrix{\be:Y~\ar[r]\ar@{=}[d]& T_{n-1}\ar[r]\ar[d] &\cdots\ar[r] & T_0\ar[r]\ar[d] & X\ar@{=}[d]\\ \be':Y ~\ar[r] & T'_{n-1}\ar[r]& \cdots\ar[r]& T'_0\ar[r]& X.}\]}Then any $\ruf$ of $\be'$, will be an $\ruf$ of $\be$, as well. So, for any conflation $\ga$, we will have $\bar{\ga}\circ\bar{\be}=\bar{\ga}\circ\bar{\be'}$, where the composition makes sense. On the other hand,
since for a given morphism $f$, the equality $\be f=\be'f$ holds, one may infer that $\bar{\be}\circ\bar{\al}=\bar{\be'}\circ\bar{\al}$, for all conflations $\al\in\Ext^n(Z, X)$.
These facts, in conjunction with \cite[Chapter VII, Proposition 3.1]{mit} guarantee that the
operator $``\circ^{,,}$ is compatible with the equivalence classes in $\Ext^n(X, Y)$. On the other hand, as we have mentioned in \ref{pp}, $\p$ is an $\H$-submodule of $\Ext^n$. Hence $\Ext^n/{\p}$ will be an $\H$-bimodule.
\end{rem}

Taking all of the previous results together, leads us to deduce that the composition operator $``\circ"$, introduced in Definition \ref{compo}, is well-defined. Indeed we have the next result.
\begin{theorem}\label{circ}The composition $``\circ"$ is well-defined.
\end{theorem}

\begin{prop}\label{ass}The composition operator $``\circ"$ is associative.
\end{prop}
\begin{proof}
Assume that $M, N, K, L\in\C$. Fix unit conflations $\delta_N\in\Ext^n(N, \syz^nN)$ and $\delta_K\in\Ext^n(K, \syz^nK)$ and let $\bar{\al}\in\Ext^n(M, \syz^nN)/{\p}$, $\bar{\be}\in\Ext^n(N, \syz^n K)/{\p}$ and $\bar{\ga}\in\Ext^n(K, \syz^nL)/{\p}$. We would like to show that $(\bar{\ga}\circ\bar{\be})\circ\bar{\al}=\bar{\ga}\circ(\bar{\be}\circ\bar{\al})$. Let us first compute $\bar{\ga}\circ(\bar{\be}\circ\bar{\al})$. In this direction, assume that $\al=\delta f$ is an $\ruf$ of $\al$ and suppose that $[\delta_Na_1, \delta b_1]$ is a co-angled pair. So by definition, $\bar{\be}\circ\bar{\al}=\overline{((\be a_1)b_1^{-1})f}$. For the simplicity, set $\be':=(\be a_1)b_1^{-1}$ and take an $\ruf$, $\be'=\delta' g'$ of $\be'$, then, $\be'f=\delta'(g'f)$ is an $\ruf$ of $\be'f$. Now, assume that $[\delta_Ka_2, \delta'b_2]$ is a co-angled pair, one has the equality $\bar{\ga}\circ(\bar{\be}\circ\bar{\al})=\overline{((\ga a_2)b_2^{-1})g'f}$. Next we calculate $(\bar{\ga}\circ\bar{\be})\circ\bar{\al}$. Suppose that $\be=\delta''g$ is an $\ruf$ of $\be$ and $[\delta_Ka_3, \delta''b_3]$ is a co-angled pair, and so, $\bar{\ga}\circ\bar{\be}=\overline{((\ga a_3)b_3^{-1})g}$. By our assumption, $\be a_1=_{\p}\be'b_1$. {So  Lemma \ref{corwell} yields that $0=\bar{\ga}\circ\overline{(\be a_1-\be'b_1)}=\bar{\ga}\circ\overline{\be a_1}-\bar{\ga}\circ\overline{\be'b_1}$. As $\be=\delta''g$, we have $\be a_1=\delta''(ga_1)$. Thus, considering the co-angled pair $[\delta_Ka_3, \delta''b_3]$, one has $\bar{\ga}\circ\overline{\be a_1}=\overline{((\ga a_3)b_3^{-1})g a_1}$. Similarly, since $\be'b_1=\delta'(g'b_1)$, we obtain that $\bar{\ga}\circ\overline{\be' b_1}=\overline{((\ga a_2)b_2^{-1})g' b_1}$.} Consequently, $((\ga a_3)b_3^{-1})g a_1=_{\p}{((\ga a_2)b_2^{-1})g'b_1}$, meaning that $(\bar{\ga}\circ\bar{\be})a_1=\overline{(((\ga a_2)b_2^{-1})g')b_1}$. Thus, by considering the co-angled pair $[\delta_Na_1, \delta b_1]$, one may deduce that $(\bar{\ga}\circ\bar{\be})\circ\bar{\al}=\overline{(((\ga a_2)b_2^{-1})g')f}$. Therefore, $(\bar{\ga}\circ\bar{\be})\circ\bar{\al}=\bar{\ga}\circ(\bar{\be}\circ\bar{\al})$, as needed.
\end{proof}

\begin{cor}\label{cor100}Let $M\in\C$ and fix a unit conflation $\delta_M\in\Ext^n(M, \syz^nM)$. Then $(\Ext^n(M, \syz^nM)/{\p}, +, \circ)$ has a ring structure with the identity element $\bar{\delta}_M$. Moreover, for any unit conflation $\delta\in\Ext^n(M, \syz^nM)$, $\bar{\delta}$ is invertible.
\end{cor}
\begin{proof}According to Theorem \ref{circ} and Propositions \ref{ass} and \ref{srt1}, we only need to show that $\bar{\delta}_M$ is the unit element of $\Ext^n(M, \syz^nM)/{\p}$. To see this, take $\bar{\ga}\in\Ext^n(M, \syz^n M)/{\p}$. Since $\delta_M=\delta_M 1_M$ is an $\ruf$ of $\delta_M$, one has $\bar{\ga}\circ\bar{\delta}_M =\bar{\ga}$. Now suppose that $\ga=\delta f $ is an $\ruf$ of $ \ga $ and $[\delta_Ma, \delta b]$ is a co-angled pair. Consequently, $\bar{\delta}_M\circ\bar{\ga}=\overline{((\delta_Ma)b^{-1})f}=\overline{\delta f}= \bar{ \ga} $, gives the first assertion.
For the second assertion, assume that $\delta\in\U_n(M)\cap\Ext^n(M, \syz^nM)$. Consider a co-angled pair $\delta_M\st{a}\llf\delta'\st{b}\lrt\delta$. Take $\al\in\Ext^n(M, \syz^nM)$ such that $(\al a)b^{-1}=_{\p}\delta_M$, i.e., $\al a=\delta_Mb$. According to Lemma \ref{unit}, $\al$ is a unit conflation. As $\al=\al 1_M$ is an $\ruf$ of $\al$, $\bar{\delta}\circ\bar{\al}=\overline{((\delta b)a^{-1})1}_M=\bar{\delta}_M$. Similarly, since $\delta=\delta 1_M$ is an $\ruf$ of $\delta$, by the definition, we have $\bar{\al}\circ\bar{\delta}=\overline{((\al a)b^{-1})1}_M=\bar{\delta}_M$. Hence $\bar{\al}$ is the inverse of $\bar{\delta}$, and so, the proof is finished.
\end{proof}

We close this section with the following interesting result.

\begin{prop}\label{ringhom} Let $M\in\C$ and fix a unit conflation $\delta_M\in\Ext^n(M, \syz^nM)$. Then there exists a ring homomorphism $\varphi :\hom_{ \C}(M,M) \rightarrow \Ext^n(M, \syz^n M)/{\p} $ with $ \varphi(f)=\overline{ \delta_Mf}$ such that for any quasi-invertible morphism $f$, $\varphi(f)$ is an invertible element of $\Ext^n(M, \syz^n M)/{\p}$, and $\varphi(f)=0$, if $f$ is an $n$-$\Ext$-phantom morphism.
\end{prop}
\begin{proof} Assume that $f, g\in\hom_{\C}(M, M).$ According to \cite[Chapter VII, Lemma 3.2]{mit}, $ \delta_M (f+g)= \delta_Mf + \delta_Mg$, meaning that $ \varphi$ is a morphism of abelian groups. Moreover, we have $\varphi(g)\circ\varphi(f)=\overline{\delta_Mg}\circ\overline{\delta_Mf}=\overline{(\delta_Mg)f}=\overline{\delta_M(gf)}=\varphi(gf)$. Note that the second equality follows from the fact that, one may choose $[\delta_M1_M, \delta_M 1_M]$ as a co-angled pair. Finally, $\varphi(1_M)=\overline{\delta_M1}_M=\bar{\delta}_M$, and then, $\varphi$ is a ring homomorphism. Next assume that $f\in\si$. Since $\delta_M\in\U_n(M)$, by Lemma \ref{unit}, $\delta_Mf \in \U_n(M)$. Indeed, $\delta_Mf\in\U_n(M)\cap\Ext^n(M, \syz^nM)$. So Corollary \ref{cor100} implies that $\varphi(f)$ is invertible. Moreover, the last assertion holds true, because of Proposition \ref{three}. Hence the proof is finished.
\end{proof}

\section{An equivalence relation on $\Ext^n/{\p}$}
Let $M$ and $N$ be objects of $\C$. This section is devoted to defining an equivalence relation on the class $\bigcup_{ \delta_N}(\Ext^n( M, \syz^n N)/{\p}, \delta_N)$.
It will be observed that this relation is compatible with the composition operator $``\circ"$, as well as the operator $``+"$.

\begin{dfn}\label{rel}
Assume that $M,N\in\C$ and $(\bar{\ga}, \delta_N), (\bar{\ga'}, \delta'_N)$ are arbitrary objects of $\bigcup_{ \delta_N\in\U_n(N)}(\Ext^n(M, \syz^n N)/{\p}, \delta_N)$. We write $(\bar{\ga}, \delta_N)\thicksim(\bar{\ga'}, \delta'_N)$, if $a\ga-b\ga'$ is a $\p$-conflation, where $[a\delta_N, b\delta'_N]$ is an angled pair, which exists by Proposition \ref{pro100}.
\end{dfn}

In the sequel, we show that $``\thicksim"$ is an equivalence relation.

\begin{theorem}
$``\thicksim"$ is an equivalence relation.
\end{theorem}
\begin{proof}
Since the reflexivity and symmetry hold trivially, we only need to show the transitivity. To this end,
assume that $(\bar{\ga}, \delta)\sim (\bar{\ga'}, \delta')$ and $(\bar{\ga'}, \delta')\sim (\bar{\ga''}, \delta'')$. We shall prove that $(\bar{\ga}, \delta)\thicksim (\bar{\ga''}, \delta'')$. By our assumption, there are angled pairs $\delta'\st{a_1}\lrt\delta_1\st{b_1}\llf\delta$ and $\delta'\st{a_2}\lrt\delta_2\st{b_2}\llf\delta''$ such that $a_1\ga'-b_1\ga$ and $a_2\ga'-b_2\ga''$
are $\p$-conflations. Now considering an angled pair $\delta_1\st{a_3}\lrt\delta_3\st{b_3}\llf\delta_2$, which exists by
part (2) of Proposition \ref{pro100}, we will get the following commutative diagram of angled pairs: \[\xymatrix{&&\delta'\ar[ld]_{a_1}\ar[rd]^{a_2}&& \\ \delta\ar[r]^{b_1}&\delta_1~\ar[r]^{a_3}& \delta_3& \delta_2\ar[l]_{b_3}&\delta''\ar[l]_{b_2} .}\] Since $(a_3a_1-b_3a_2)\delta'=0$, a dual argument given in the proof of Proposition \ref{three}, would imply that $(a_3a_1-b_3a_2)\ga'$ is a $\p$-conflation. Hence, in view of our hypothesis, we may deduce that $(a_3b_1)\ga-(b_3b_2)\ga''$ is a $\p$-conflation. Since the class $\si$ is closed under composition, $[(a_3b_1)\delta, (b_3b_2)\delta'']$ will be also an angled pair. Consequently, $(\bar{\ga}, \delta)\thicksim (\bar{\ga''}, \delta'')$, as desired.
\end{proof}

\begin{prop}\label{ind}Let $M, N$ be objects of $\C$ and fix a unit conflation $\delta_N\in\Ext^n(N, \syz^nN)$. Then for any $(\bar{\ga'}, \delta'_{N})\in(\Ext^n( M, {\syz'}^nN)/{\p}, \delta'_N)$, there is a unique object $(\bar{\ga}, \delta_N)\in(\Ext^n(M, \syz^nN)/{\p}, \delta_N)$ which is equivalent to $(\bar{\ga'}, \delta'_N)$.
\end{prop}
\begin{proof}
Take an angled pair $\delta_N\st{a}\lrt\delta''\st{b}\llf\delta'_N$. Since $b\ga'\in\Ext^n(M, {\syz''}^nN)$ and $(\syz^nN\st{a}\rt{\syz''}^nN)\in\si$, in view of Corollary \ref{div}(2), there exists $\ga\in\Ext^n(M, \syz^nN)$ such that $a\ga-b\ga'$ is a $\p$-conflation. Moreover, uniqueness of $(\bar{\ga}, \delta_N)$ follows from Proposition \ref{nul}. So the proof is finished.
\end{proof}

\begin{rem}
Assume that $M,N\in\C$ and fix a unit conflation $\delta_N\in\Ext^n(N, \syz^nN)$. Then it follows from Proposition \ref{ind} that $ (\Ext^n(M, \syz^n N)/{\p}, \delta_N)$ is equal to the set of all equivalence classes of $\bigcup_{{\delta'}_{N}}(\Ext^n(M, {\syz'}^n N)/{\p}, \delta'_N)$ modulo the equivalence relation $``\thicksim"$.
\end{rem}

The following easy observation is needed in the subsequent proposition.
\begin{lem}\label{pushco}Let $\delta, \delta'\in\U^n(\syz^nM)$ and $\delta\st{a}\llf\delta''\st{a'}\lrt\delta'$ be a co-angled pair. Let $b:\syz^nM\rt{\syz'}^nM$ be a morphism in $\si$.
Then $b\delta\st{a}\llf b\delta''\st{a'}\lrt b\delta'$ is also a co-angled pair.
\end{lem}
\begin{proof}Since $b\in\si$, by Lemma \ref{unit}, the push-out of any unit conflation along $b$ is also a unit conflation.
By the hypothesis, $\delta a=\delta''=\delta'a'$, and so, by Remark \ref{pp1}, we obtain the co-angled pair $b\delta\st{a}\llf b\delta''\st{a'}\lrt b\delta'$, as desired.
\end{proof}

\begin{prop}\label{comequ}The equivalence relation $``\thicksim"$ is compatible with the composition $``\circ"$.
\end{prop}
\begin{proof}Let $(\bar{\ga}, \delta_M)\in(\Ext^n(K, \syz^nM)/{\p}, \delta_M)$ and $(\bar{\ga'}, \delta'_M)\in(\Ext^n(K, {\syz'}^nM)/{\p}, \delta'_M)$ such that $(\bar{\ga}, \delta_M)\thicksim(\bar{\ga'}, \delta'_M)$. Suppose that $(\bar{\be}, \delta_N)\in(\Ext^n(M, \syz^nN)/{\p}, \delta_N)$. We must show that $(\bar{\be}\circ\bar{\ga}, \delta_N)\thicksim(\bar{\be}\circ\bar{\ga'}, \delta_N)$. By our hypothesis, there is an angled pair $\delta_M\st{a}\lrt\delta''_M\st{b}\llf\delta'_M$
such that $a\ga-b\ga'$ is a $\p$-conflation. Set $\ga'':=a\ga$. So $(\bar{\ga''}, \delta''_M)\in(\Ext^n(K, {\syz''}^nM)/{\p}, \delta''_M)$. Hence in order to obtain the result, it suffices to show that $(\bar{\be}\circ\bar{\ga}, \delta_N)\thicksim(\bar{\be}\circ\bar{\ga''}, \delta_N)$. To do this, assume that $\ga=\delta_1f$ is an $\ruf$ of $\ga$. Take a co-angled pair $[\delta_M a_1, \delta_1 b_1]$. As $a\in\si$, by Lemma \ref{pushco}, $[(a\delta_M) a_1, (a\delta_1) b_1]$ is also a co-angled pair. Since $\ga=\delta_1f$, one has $\ga''=a\ga=a(\delta_1f)=(a\delta_1)f$, meaning that $\ga''=(a\delta_1)f$ is an $\ruf$ of $\ga''$. Consequently, $\bar{\be}\circ\bar{\ga}=\overline{((\be a_1){b_1}^{-1})f}=\bar{\be}\circ\bar{\ga''}$, and so, $(\bar{\be}\circ\bar{\ga}, \delta_N)\thicksim(\bar{\be}\circ\bar{\ga''}, \delta_N)$, as desired.

Next consider $(\bar{\be'}, \delta'_N)\in(\Ext^n(M, {\syz'}^nN)/{\p}, \delta'_N)$ such that $(\bar{\be}, \delta_N)\thicksim(\bar{\be'}, \delta'_N)$. We would like to show that $(\bar{\be}\circ\bar{\ga}, \delta_N)\thicksim(\bar{\be'}\circ\bar{\ga}, \delta'_N)$.
By the hypothesis, there is an angled pair $[a\delta_N, a'\delta'_N]$ such that $a\be-a'\be'$ is a $\p$-conflation, and so, $a\bar{\be}=a'\bar{\be'}$. Now from the defnition of the composition $``\circ"$, one may deduce that $a(\bar{\be}\circ\bar{\ga})=a'(\bar{\be'}\circ\bar{\ga})$, meaning that $(\bar{\be}\circ\bar{\ga}, \delta_N)\thicksim(\bar{\be'}\circ\bar{\ga}, \delta'_N)$. Thus, the proof is finished.
\end{proof}

\begin{lem}The equivalence relation $``\thicksim"$ is compatible with $``+"$.
\end{lem}
\begin{proof}Assume that $(\bar{\ga}, \delta_N), (\bar{\ga'}, \delta_N)\in(\Ext^n(M, \syz^nN)/{\p}, \delta_N)$ and $(\bar{\be}, \delta'_N), (\bar{\be'}, \delta'_N)\in(\Ext^n(M, {\syz'}^n N)/{\p}, \delta'_N)$ such that $(\bar{\ga}, \delta_N)\thicksim(\bar{\be}, \delta'_N)$ and $(\bar{\ga'}, \delta_N)\thicksim(\bar{\be'}, \delta'_N)$. We shall prove that $(\overline{\ga+\ga'}, \delta_N)\thicksim(\overline{\be+\be'}, \delta'_N)$.
Since $(\bar{\ga}, \delta_N)\thicksim(\bar{\be}, \delta'_N)$ and $(\bar{\ga'}, \delta_N)\thicksim(\bar{\be'}, \delta'_N)$, by the definition, $a\ga-b\be$ and $a\ga'-b\be'$ are $\p$-conflations, where $[a\delta_N, b\delta'_N]$ is an angled pair. Hence, $a(\ga+\ga')-b(\be+\be')$ is a $\p$-conflation, as well. Consequently, $(\overline{\ga+\ga'}, \delta_N)\thicksim (\overline{\be+\be'}, \delta'_N)$, as required.
\end{proof}

\begin{prop}\label{val}Let $(\bar{\ga}, \delta_N)\thicksim(\bar{\ga'}, \delta'_N)$. Then $\ga$ is $\p$-conflation if and only if $\ga'$ is so.
\end{prop}
\begin{proof}Assume that $\ga$ is a $\p$-conflation. We need to show that $\ga'$ is a $\p$-conflation, as well. By the hypothesis, $a\ga-b\ga'$ is a $\p$-conflation, where $[a\delta_N, b\delta'_N]$ is an angled pair. Since $\ga$ is a $\p$-conflation, Proposition \ref{nul} yields that $a\ga$ is also a $\p$-conflation. So using the fact that $a\ga-b\ga'$ is a $\p$-conflation, one may infer that the same is true for $b\ga'$. Hence, another use of Proposition \ref{nul}, guarantees that $\ga'$ is also a $\p$-conflation. Since the reverse implication is obtained similarly, we skip it. So the proof is finished.
\end{proof}

{
\section{Phantom stable categories}
Inspired by the stabilization of a Frobenius category, we introduce and study the notion of the phantom stable category of an $n$-Frobenius category. We begin with the following motivating observation.

\begin{s}\label{s1s1}
Let $\C'$ be a Frobenius category and let $\mathcal{I}$ be the ideal consisting of all morphisms factoring through projective objects. Assume that $\C'/{\mathcal{I}}$ is the stable category of $\C'$. So we have the natural functor $\pi:\C'\lrt\C'/{\mathcal{I}}$ such that for any morphism $f$ that its kernel and cokernel are both projective, $\pi(f)$ is an isomorphism and for any $g\in\mathcal{I}$, $\pi(g)=0$. It is easily seen that the pair $(\C'/{\mathcal{I}}, \pi)$ has a universal property with respect to these conditions. This fact is our idea to introduce the notion of the {\it phantom stable category} of an $n$-Frobenius category, for any $n\geq 0$. To be precise, let $\C$ be an $n$-Frobenius category and let $\Ext^n$ be all equivalence classes of conflations of length $n$. Assume that $\p$ is a subfunctor of $\Ext^n$ consisting of all
conflations of length $n$ which are obtained as pull-back of conflations along morphisms of the form $M\rt P$, for some $M\in\C$ and $P\in n$-$\proj\C$. Assume that $\si$ is the class of all morphisms acting as invertible on $\Ext^{n+1}$. We introduce the additive category $\C_{\p}$ and an additive functor $T:\C\rt\C_{\p}$ with $T(s)$ is an isomorphism, for any quasi-invertible morphism $s$ and $T(f)=0$ for all $n$-$\Ext$-phantom morphisms $f$, which has universal property with respect to these conditions. Indeed, we have the following definition.
\end{s}

\begin{dfn}We say that a couple $(\C_{\p}, T)$, where $\C_{\p}$ is an additive category and $T:\C\lrt\C_{\p}$ is a covariant additive functor, is the {\it phantom stable category of $\C$}, if:\\ (1) $T(s)$ is an isomorphism in $\C_{\p}$, for any quasi-invertible morphism $s$.\\
(2) For any $n$-$\Ext$-phantom morphism $\varphi$, $T(\varphi)=0$ in $\C_{\p}$.\\(3) Any covariant additive functor $T':\C\lrt\D$ satisfying the conditions (1) and (2), factors in a unique way through $T$.
\end{dfn}

In the following, we show the existence of the phantom stable category of $\C$. First, we quote a couple of auxiliary results.

\begin{lem}\label{iso}Let $(\bar{\ga}, \delta_N)\in(\Ext^n(M, \syz^nN)/{\p}, \delta_N)$ such that $\ga$ is a unit conflation. Then for any $\delta_M\in\U_n(M)$, there exists $(\bar{\ga'}, \delta_M)\in(\Ext^n(N, \syz^nM)/{\p}, \delta_M)$ such that $\bar{\ga'}\circ\bar{\ga}=\bar{\delta}_M$ and $\bar{\ga}\circ\bar{\ga'}=\bar{\delta}_N$.
\end{lem}
\begin{proof}Take a co-angled pair $[\ga a, \delta_Nb]$ and set $\ga':=(\delta_Ma)b^{-1}$. So considering an $\ruf$ $\ga=\ga 1_M$ of $\ga$, we have $\bar{\ga'}\circ\bar{\ga}=\overline{(\ga'b)a^{-1}}=\overline{(((\delta_Ma)b^{-1})b)a^{-1}}=\bar{\delta}_M$. Next we show that $\bar{\ga}\circ\bar{\ga'}=\bar{\delta}_N$. Since $\ga'=(\delta_Ma)b^{-1}$, {by Lemma \ref{unit}, $\ga'$ is a unit conflation and so} we may take the co-angled pair $[\ga'b, \delta_Ma]$. Now considering $\ga'=\ga' 1_N$ as an $\ruf$ of $\ga'$, and  Theorem \ref{welldef}, $\bar{\ga}\circ\bar{\ga'}=\overline{(\ga a)b^{-1}}=\bar{\delta}_N$, would give the desired result.
\end{proof}

\begin{rem}\label{00}Assume that $(\bar{\ga}, \delta_N)\thicksim(\bar{\ga'}, \delta'_N)$. So, by the definition, $a\ga-b\ga'$ is a $\p$-conflation, where $[a\delta_N, b\delta'_N]$ is an angled pair. Thus for any morphism $f$, $(a\ga-b\ga')f=a(\ga f)-b(\ga'f)$ will be also a $\p$-conflation, meaning that $(\overline{\ga f}, \delta_N)\thicksim(\overline{\ga'f}, \delta'_N)$.
\end{rem}

The result below is the main theorem of this section.
\begin{theorem}\label{thmst}The phantom stable category $(\C_{\p}, T)$ of $\C$ exists.
\end{theorem}
\begin{proof}We define the category $\C_{\p}$ as follows: the objects of $\C_{\p}$ are the same as the  objects of $\C$. For any two objects $M, N\in\C$, first fix a unit conflation $\delta_N\in\Ext^n(N, \syz^nN)$ and set $\hom_{\C_{\p}}(M, N):=(\Ext^n(M, \syz^nN)/{\p}, \delta_N)$ modulo the equivalence relation $``\thicksim"$.
Assume that $(\bar{\ga}, \delta_N)\in(\Ext^n(M, \syz^nN)/{\p}, \delta_N)$ and $(\bar{\be}, \delta_K)\in(\Ext^n(N, \syz^nK)/{\p}, \delta_K)$. We define $(\bar{\be}, \delta_K)\circ(\bar{\ga}, \delta_N):=(\bar{\be}\circ\bar{\ga}, \delta_K)$, which is well-defined, thanks to Proposition \ref{comequ}. According to Proposition \ref{ass}, the composition operator $``\circ"$ is associative. For a given object $M\in\C_{\p}$, fix a unit conflation $\delta_M\in\Ext^n(M, \syz^nM)$. We claim that $1_M=(\bar{\delta}_M, \delta_M)$.
Indeed, it is evident that for any $(\bar{\ga}, \delta_N)\in(\Ext^n(M, \syz^nN)/{\p}, \delta_N)$, $(\bar{\ga}, \delta_N)\circ(\bar{\delta}_M, \delta_M)=(\bar{\ga}, \delta_N)$. Now, take a morphism
$(\bar{\al}, \delta'_M)\in(\Ext^n(N, {\syz'}^nM)/{\p}, \delta'_M)$. Applying Propositions \ref{comequ} and \ref{ind}, allows us to assume that $\delta'_M=\delta_M$, and then, one may easily see that $(\bar{\delta}_M, \delta_M)\circ(\bar{\al}, \delta_M)=(\bar{\al}, \delta_M)$. Thus $1_M=(\bar{\delta}_M, \delta_M),$ as claimed. So $\C_{\p}$ is a category. Clearly, $\C_{\p}$ is closed under finite direct sums, because the same is true for $\C$, and for any two objects $M, N$, $\hom_{\C_{\p}}(M, N)$ is an abelian group. Moreover, Propositions \ref{srt1} and \ref{comequ} guarantee that the composition is bilinear, that is, the composition distributes over addition. Consequently, $\C_{\p}$ is an additive category.

Now let us define an additive covariant functor $T:\C\lrt\C_{\p}$. For any object $M\in\C$, we let $T(M)=M$. For given $M, N\in\C$,  consider a unit conflation $\delta_N$, and define the morphism $T_{M,N}:\hom_{\C}(M, N)\lrt\hom_{\C_{\p}}(M, N)$, as $T(f):=(\overline{\delta_Nf}, \delta_N)$, for any morphism $f\in\hom_{\C}(M, N)$. It should be pointed out that, Proposition \ref{ind} together with Remark \ref{00}, would imply that $T$ is well-defined.
By our definition, $T(1_N)=(\bar{\delta}_N, \delta_N)=1_{T(N)}$. Furthermore, it is easily seen that for any two composable morphisms $M\st{f}\rt N\st{g}\rt K$ in $\C$, the equalities $\overline{\delta_Kg}\circ\overline{\delta_Nf}=\overline{(\delta_Kg)f}=\overline{\delta_K(gf)}$ hold true, implying that $T(gf)=T(g)\circ T(f)$. Thus $T$ is a covariant functor. Since $\C$ and $\C_{\p}$ are additive categories and $T$ preserves finite direct sums, it will be an additive functor. Assume that $f:M\rt N$ lies in $\si$. By Lemma \ref{unit}, $\delta_Nf$ is a unit conflation, and so, Lemma \ref{iso} forces $T(f)$ to be an isomorphism.
Suppose that $h: X\rt N$ is an $n$-$\Ext$-phantom morphism. So, by Proposition \ref{three}, $\delta_Nh$ will be a $\p$-conflation, implying that $T(h)=0$.
Finally, assume that $T':\C\lrt\D$ is a covariant additive functor such that $T'(f)$ is an isomorphism, for any $f\in\si$ and for any $n$-$\Ext$-phantom morphism $g$, $T'(g)=0$. To complete the proof, we must prove that there exists a unique additive functor $F:\C_{\p}\lrt\D$ such that $FT=T'$. To do this, we shall define the functor $F:\C_{\p}\lrt\D$ as follows: for any object $X\in\C_{\p}$, write $F(X)=T'(X)$. Also, for any morphism $(\bar{\ga}, \delta_N)\in\hom_{\C_{\p}}(M,N)=(\Ext^n(M, \syz^nN)/{\p}, \delta_N)$, take an $\ruf$ $\ga=\delta' f'$ of $\ga$. So, considering a co-angled pair $[\delta_Na_1, \delta' b_1]$, we define $F((\bar{\ga}, \delta_N))=T'(a_1)T'(b_1)^{-1}T'(f')$, with $(T'(b_1))^{-1}$ being the inverse of $T'(b_1)$. We would like to show that $F$ is well-defined. In this direction, first we prove that $F((\bar{\ga}, \delta_N))$ is independent of the choice of $\ruf$ of $\ga$. Assume that $\ga=\delta'' f''$ is also another $\ruf$ of $\ga$. Now taking a co-angled pair $[\delta_N a_2, \delta''b_2]$, we shall show that $T'(a_1)T'(b_1)^{-1}T'(f')=T'(a_2)T'(b_2)^{-1}T'(f'')$. According to the proof and notations of Theorem \ref{welldef}, we get the following commutative diagram in $\C$:
{\footnotesize \[\xymatrix{&N& \\ N_1~\ar[d]_{b_1}\ar[ur]^{a_1}& N_4\ar[l]_{a_4}\ar[r]^{b_4}\ar[d]_{c_4}& N_2\ar[d]_{b_2}\ar[ul]_{a_2}\\ N'~ & N_3\ar[l]_{a_3}\ar[r]^{b_3} & N'',}\]}where $N'$ (resp. $N''$) is the right end term of the unit conflation $\delta'$ (resp. $\delta''$), $N_i$ for any $i$, is the right end term of $\delta_i$ and all morphisms are co-induced by identity over $\syz^nN$ . Moreover, by Proposition \ref{coin} and combining with Lemma \ref{ds}, there is a morphism $h: M\rt N_3$ such that $f'=a_3h$ and $f''=b_3h$. Namely, one may have the following commutative diagram in $\C$: {\footnotesize\[\xymatrix{N'~& N_3\ar[l]_{a_3}\ar[r]^{b_3}& N''\\ & M\ar[u]^{h}\ar[ul]^{f'}\ar[ur]_{f''} .& }\]}Since $T'$ is an additive functor, applying to the above diagrams, gives us commutative diagrams in $\D$. This, in conjunction with the fact that $T'(a)$ is an isomorphism, for any $a\in\si$, would imply that $T'(a_1)T'(b_1)^{-1}T'(f')=T'(a_2)T'(b_2)^{-1}T'(f'')$, as claimed. Finally, we show that if $(\bar{\ga}, \delta_N)\thicksim(\bar{\ga'}, \delta'_N)$, then $F((\bar{\ga}, \delta_N))=F((\bar{\ga'}, \delta'_N))$. Assume that $\ga=\delta f$ and $\ga'=\delta'f'$ are $\ruf$s of $\ga$ and $\ga'$, respectively. So, taking co-angled pairs $[\delta_Na_1, \delta b_1]$ and $[\delta'_Na_2, \delta'b_2]$, we have to show that $T'(a_1)T'(b_1)^{-1}T'(f)=T'(a_2)T'(b_2)^{-1}T'(f')$. By our hypothesis, there is an angled pair $\delta_N\st{a}\lrt\delta''_N\st{b}\llf\delta'_N$
such that $a\ga-b\ga'$ is a $\p$-conflation. Set $\ga'':=a\ga=_{\p}b\ga'$. So considering $(\bar{\ga''}, \delta''_N)\in(\Ext^n(M, {\syz''}^nN)/{\p}, \delta''_N)$, we infer that $\ga''=(a\delta)f=(b\delta')f'$ are two $\ruf$s of $\ga''$. According to Lemma \ref{pushco}, one has the co-angled pairs, $[\delta''_Na_1,(a\delta)b_1]$ and $[\delta''_Na_2, (b\delta')b_2]$. Since  the definition of $F$ is independent of the choice of $\ruf$ of $\ga''$, we infer that $T'(a_1)T'(b_1)^{-1}T'(f)=T'(a_2)T'(b_2)^{-1}T'(f')$, as desired. {For a given object $M\in\C_{\p}$, we clearly have the equalities, $F(1_M)=(\bar{\delta}_M, \delta_M)=T'(1_M)=1_{T'(M)}=1_{F(M)}$. Next assume that $(\bar{\ga}, \delta_N)\in(\Ext^n(M, \syz^nN)/{\p}, \delta_N)$ and $(\bar{\be}, \delta_K)\in(\Ext^n(N, \syz^nK)/{\p}, \delta_K)$. We have to show that $F((\bar{\be}, \delta_K))\circ F((\bar{\ga}, \delta_N))=F((\bar{\be}\circ\bar{\ga}, \delta_K))$. Suppose that $\ga=\delta_{N'}f$ and $\be=\delta_{K'}$ are $\ruf$s of $\ga$ and $\be$, respectively. Thus, taking co-angled pairs $[\delta_Na, \delta_{N'}b]$ and $[\delta_Kc, \delta_{K'}e]$, we get the equalities; $F((\bar{\ga}, \delta_N))=T'(a)T'(b)^{-1}T'(f)$ and $F((\bar{\be}, \delta_K))=T'(c)T'(e)^{-1}T'(g)$, and so, $F((\bar{\be}, \delta_K))\circ F((\bar{\ga}, \delta_N))=T'(c)T'(e)^{-1}T'(g)T'(a)T'(b)^{-1}T'(f)$. On the other hand, $\bar{\be}\circ\bar{\ga}=\overline{((\be a)b^{-1})f}.$ Set $\et:=(\be a)b^{-1}$. So one has $(\overline{\be a}, \delta_K)=(\overline{\et b}, \delta_K)$. Taking an $\ruf$ $\et=\delta_{K''}h$ of $\et$, one gets $\et f=\delta_{K''}(hf)$ is an $\ruf$ of $\et f$. Therefore, considering a co-angled pair$[\delta_K s, \delta_{K''}s]$, we have that $F(\bar{\be}\circ\bar{\ga}, \delta_K)=F(\overline{\et f}, \delta_K)=T'(s)T'(t)^{-1}T'(h)T'(f)$. Moreover, the well-definedness of $F$ yeilds that $F((\overline{\be a}, \delta_K))=F((\overline{\et b}, \delta_K))$, and so, $T'(c)T'(e)^{-1}T'(g)T'(a)=T'(s)T'(t)^{-1}T'(h)T'(b)$. This would imply that $F((\bar{\be}, \delta_K))\circ F((\bar{\ga}, \delta_N))=F((\bar{\be}\circ\bar{\ga}, \delta_K))$, as desired.} It is clear that $FT=T'$ and also uniqueness of $F$ is obvious. So the proof is complete.
\end{proof}

Assume that $\C'$ is a Frobenius category (or 0-Frobenius category in our sense). Then it follows from the proof of the above theorem that the phantom stable category $(\C'_{\p}, T)$ is indeed $(\C'/{\I}, \pi)$, which has been mentioned in \ref{s1s1}. Namely, in the case $n=0$, the phantom stable category is actually the classical stable category of a Frobenius category.

From now on, to simplify the notation, we shall denote $\hom_{\C_{\p}}(-, -)$ by $\C_{\p}(-, -)$.

\begin{lem}\label{lem1}Let $\C'$ be a full subcategory of $\C$ which is {closed under extensions and kernels of epimorphisms}. Assume that
\begin{enumerate}\item $\C'$ is an $n$-Frobenius category. \item $n$-$\proj\C'\subseteq n$-$\proj\C$. \item $\C$ has enough $n$-$\proj\C'$.
\end{enumerate}Then for any two objects $M, N$ in $\C'$, ${\C_{\p}}(M, N)={\C'_{\p}}(M, N)$.
\end{lem}
\begin{proof}
Assume that $M, N$ are arbitrary objects of $\C'$ and take a morphism $(\bar{\ga}, \delta_N)\in{\C_{\p}}(M, N)=(\Ext^n_{\C}(M, \syz^nN)/{\p}, \delta_N)$. {Since $n$-$\proj\C'\subseteq n$-$\proj\C$, we may assume that $\syz^nN\in\C'$. Moreover}, using the fact that $\C$ has enough $n$-$\proj\C'$, one may follow the argument given in the proof of Lemma \ref{gencog}, and get the following commutative diagram:
{\footnotesize \[\xymatrix{\ga':\syz^nN~\ar[r]\ar@{=}[d]& H'\ar[r]\ar[d]_{b_{n-1}}& P_{n-2}\ar[r]\ar[d]_{b_{n-2}} &\cdots \ar[r]& P_0\ar[r]\ar[d]_{b_0} & M\ar@{=}[d]\\ \ga:\syz^nN~\ar[r] &X_{n-1}\ar[r]& X_{n-2}\ar[r]&\cdots\ar[r]& X_0\ar[r] & M,}\]}such that all $P_i^,$s belong to $n$-$\proj\C'$. Since $\C'$ is a full subcategory of $\C$ which is closed under extensions and kernels of epimorphisms, $\ga'$ will be a conflation in $\C'$. This, in turn, implies that $\ga$ can be considered as a conflation in $\C'$, because $\ga=\ga'$.
Now we should prove that if $\ga$ is a $\p$-conflation in $\C'$, then it is a $\p$-conflation in
$\C$ and vice versa. First, assume that $\ga$ is a $\p$-conflation in $\C'$. Since $\C'\subseteq\C$ and any $n$-projective of $\C'$ lies in $n$-$\proj\C$, we infer that $\ga$ is a $\p$-conflation in $\C$. Next, assume that $\ga$ is a $\p$-conflation in $\C$.
Thus, there is a morphism $h:P\rt\syz^nN$ with $P\in n$-$\proj\C$, and $\ep\in\Ext^n_{\C}(M, P)$ such that $\ga=h\ep$. Take a conflation $\syz P\rt Q\st{g}\rt P$, where $Q\in n$-$\proj\C'$. Evidently, $\syz P\in n$-$\proj\C$. This, in turn, yields that the morphism $\Ext^n_{\C}(M, Q)\rt\Ext^n_{\C}(M, P)$ is an epimorphism, and so, there exists $\et\in\Ext^n_{\C}(M, Q)$ such that $g\et=\ep$. As $M, Q\in\C',$ similar to the above diagram, we may assume that $\et\in\Ext^n_{\C'}(M, Q)$. Consequently, $\ga=(hg)\et$, i.e. $\ga$ is a $\p$-conflation in $\C'$, as needed.
\end{proof}

\begin{prop}\label{l}Let $\C'$ be a full subcategory of $\C$ which is closed under extensions and kernels of epimorphisms and let $k\leq n$ be non-negative integers. Assume that
\begin{enumerate}\item $\C'$ is a $k$-Frobenius category. \item $k$-$\proj\C'\subseteq n$-$\proj\C$. \item $\C$ has enough $k$-$\proj\C'$.
\end{enumerate}Consider the composition functor $T'':\C'\st{i}\rt\C\st{T}\rt\C_{\p}$, with $i$ the inclusion functor.
Then there is a unique induced fully faithful functor $F:\C'_{\p}\lrt\C_{\p}$
such that $FT'=T''$, where $(\C_{\p}, T)$ and $(\C'_{\p}, T')$ are phantom
stable categories.
\end{prop}
\begin{proof} For any $M, N\in \C'$,  consider a unit conflation $\delta_N:\syz^nN\rt P_{n-1}\rt\cdots\rt P_0\rt N$, with $P_i\in k$-$\proj\C'$, for any $i$, and so, $T''(f)=(\overline{\delta_Nf}, \delta_N)$, for any morphism $f:M\rt N$, because the definition of $T$ is independent of the choice of the unit conflation $\delta_N$. Assume that $f\in\hom_{\C'}(M, N)$ which belongs to $\si$. {In view of Lemma \ref{conf}, there is a conflation $M\rt N\oplus Q\rt P$, in which $P, Q\in k$-$\proj\C'$. Since by our assumption, $P, Q\in n$-$\proj\C$,  Corollary \ref{is} together with Remark \ref{rems}, imply that $f\in\si,$ as a morphism in $\C$, and then,} $T''(f)=T(f)$ will be an isomorphism in $\C_{\p}$. Now suppose that $f:M\rt N$ is an $n$-$\Ext$-phantom morphism in $\C'$. We shall prove that $T''(f)=0$ in $\C_{\p}$. In view of Proposition \ref{ph}, there are morphisms $N\st{a}\lf N''\st{b}\rt N'$ with $a, b\in\si$ and $h:M\rt N''$ such that $f=ah$ and $bh$ factors through an object $Q\in k$-$\proj\C'$. By our assumption, $Q\in n$-$\proj\C$, and so $bh$ will be an $n$-$\Ext$-phantom morphism in $\C$, implying that $T''(bh)=T(bh)=0$. Therefore, as $T''(b)$ is an isomorphism, we have $T''(h)=0$. Consequently, $T''(f)=T''(a)T''(h)=0$ in $\C_{\p}$, as needed. Therefore, the universal property of the phantom stable category gives rise to the existence of a unique functor $F:\C'_{\p}\rt\C_{\p}$ such that $FT'=T''$.

Now we prove that $F$ is faithful. The result for the case $k=n$ follows from Lemma \ref{lem1}. So assume that $k<n$.
Take a morphism $(\bar{\ga}, \delta)\in{\C'_{\p}}(M, N)=(\Ext^k_{\C'}(M, \syz^kN)/{\p}, \delta)$ such that $F((\bar{\ga}, \delta))=0$ in $\C_{\p}$. Suppose that $\ga=\delta' f$ is an $\ruf$ of $\ga$ and take a co-angled pair $[\delta a, \delta' b]$. Since, by the definition $F((\bar{\ga}, \delta))=T''(a)T''(b)^{-1}T''(f)$ and $T''(a)$ and $T''(b)^{-1}$ are isomorphisms, one concludes that $T''(f)\cong 0$ in $\C_{\p}$.
So taking a unit conflation $\delta_1:\syz^nN\lrt P_{n-1}\lrt\cdots\lrt P_k\lrt\syz^kN$ with $P_i^,s\in k$-$\proj\C'$, and setting $\delta'_N:=\delta_1\delta'$, we have $T''(f)=(\overline{\delta'_Nf}, \delta'_N)$. Thus one has the push-out diagram {\footnotesize\[\xymatrix{\et:P~\ar[r]\ar[d]_{g} & L\ar[d]\ar[r] &\cdots \ar[r]& P_1\ar[r]\ar@{=}[d]\ar[r]& H\ar@{=}[d]\ar[r] & M\ar@{=}[d]\\ \delta'_Nf:\syz^nN~\ar[r] & P_{n-1}\ar[r]& \cdots \ar[r] & P_1\ar[r]& H\ar[r] & M,}\]}where $P\in n$-$\proj\C$. Take a unit conflation $\syz P\rt Q\rt P$, where $Q\in k$-$\proj\C'$. Since $k<n$ and $\syz P\in n$-$\inj\C$, it is easily seen that $\Ext^n_{\C}(-, P)=0$ over $\C'$. Consequently, $\et=0$, and then, the same is true for $\delta'_Nf$. Now decomposing the unit conflation $\syz^nN\rt P_{n-1}\rt\cdots\rt P_k\rt\syz^kN$ into conflations of length 1, one may obtain the isomorphisms $\Ext^n_{\C'}(M, \syz^nN)\cong\Ext^{n-1}_{\C'}(M, \syz^{n-1}N)\cong\cdots\cong\Ext^{k+1}_{\C'}(M, \syz^{k+1}N)$. So, for $\delta_2:=\syz^{k+1}N\rt P_k\st{h}\rt\syz^kN$, we will have $\delta_2\ga=0$. Now the exact sequence $\Ext^k_{\C'}(M, P_k)\st{\h}\rt\Ext^k_{\C'}(M, \syz^kN)\st{}\rt\Ext^{k+1}_{\C'}(M, \syz^{k+1}N)$, yields that there exists $\al\in\Ext^k_{\C'}(M, P_k)$ such that $h\al=\ga$, that is, $\ga$ is a $\p$-conflation in $\C'$, and then, $(\bar{\ga}, \delta)=0$.
Finally, we show that the functor $F$ is full. Assume that $M, N\in\C'$ and fix a unit conflation $\delta_N:\syz^nN\lrt P_{n-1}\rt\cdots\rt P_0\rt N$, where $P_i\in n$-$\proj\C$, for any $i$. Suppose that $(\bar{\ga}, \delta_N)\in{\C_{\p}}(M, N)=
(\Ext^n(M, \syz^nN)/\p, \delta_N)$.
In the case $k=n$, the fullness of $F$ follows from Lemma \ref{lem1}.
So assume that $k<n$. Since any object of $k$-$\proj\C'$ is $n$-projective over $\C$, we may further
assume that all terms of the unit conflation $\delta_N$ lie in $\C'$. Therefore, one may deduce that for any $Q\in k$-$\proj\C'$ and $i>k$, $\Ext^i_{\C}(-, Q)=0$ over $\C'$. Now, decomposing the unit conflation
$\delta_1:\syz^nN\rt P_{n-1}\rt\cdots\rt P_k\rt\syz^kN$ into conflations of length 1, gives the isomorphisms
$$\Ext^n_{\C}(M, \syz^nN)\cong\Ext^{n-1}_{\C}(M, \syz^{n-1}N)\cong\cdots\cong\Ext^{k+1}_{\C}(M, \syz^{k+1}N)\cong \Ext^k_{\C}(M, \syz^kN)/\p.$$ Hence, there exists $(\bar{\ga'}, \delta'_N)\in{\C'_{\p}}(M, N)$ such that $F((\bar{\ga'}, \delta'_N))=(\overline{\delta_1\ga'}, \delta_1\delta'_N)=(\bar{\ga}, \delta_N)$, where $\delta'_N:\syz^kN\rt P_{k-1}\rt\cdots\rt P_0\rt N$. So the proof is finished.
\end{proof}

\begin{cor}\label{p}Let $n>k$ and $\C$ be also a $k$-Frobenius category. Then the phantom stable categories of $\C$ as $k$ and $n$-Frobenius, are equivalent.
\end{cor}
\begin{proof}Assume that $\C'_{\p}$ and $\C_{\p}$ denote the phantom stable categories of $\C$, as a $k$ and an $n$-Frobenius category, respectively. According to Proposition \ref{l}, there is a fully faithful functor $F:\C'_{\p}\lrt\C_{\p}$. Evidently, $F$ is also dense, and then, the proof is finished.
\end{proof}

\begin{example}\label{exfaith}
(1) With the notation of Example \ref{ex1}, we set $\C:=\G^{<\infty}$ and $\C':=\G$. Then, by Proposition \ref{l}, there exists a fully faithful functor $\C'_{\p}\rt\C_{\p}$. \\
(2) Assume that $\C$ (resp. $\C'$) is the category consisting of all syzygies of complete resolutions of $n$-projectives (resp. locally free) sheaves over $\x$. If $\C$ has enough locally free sheaves, then Proposition \ref{l} ensures the existence of
a fully faithful functor $\C'_{\p}\rt\C_{\p}$.
\end{example}
}

\begin{prop}
For a given object $M\in\C_{\p}$, the following are equivalent:
\begin{enumerate}
\item ${\C_{\p}}(-, M)=0$.
\item ${\C_{\p}}(M, -)=0$.
\item ${\C_{\p}}(M, M)=0$.
\item $M\in n$-$\proj\C$.
\end{enumerate}
\end{prop}
\begin{proof}
$(1\Rightarrow 2)$: As ${\C_{\p}}(-, M)=0$, we have ${\C_{\p}}(M, M)=0$. So $\delta_M$ is a $\p$-conflation, and in particular, the same is true for $\delta_M1_M$. So  Corollary \ref{ccoo}, yields that $(\Ext^n/{\p})1_M=0$, and particularly, $(\Ext^n(M, -)/{\p})1_M=0$, namely, $\C_{\p}(M, -)=0$.\\
$(2\Rightarrow 3)$: This is obvious.\\
$(3\Rightarrow 4)$: Take the identity morphism $(\bar{\delta}_M, \delta_M)\in\C_{\p}(M, M)$. By our assumption $\delta_M$ is a $\p$-conflation. Thus there exists a morphism $h: P\rt\syz^nM$ with
$P\in n$-$\proj\C$ and $\et\in\Ext^n(M, P)$ such that $\delta_M=h\et$. Taking an $\luf$ $\et=g\delta'_M$ of $\et$, we obtain the following push-out diagram:
\[\xymatrix{{\delta'_M:\syz'}^nM ~\ar[r] \ar[d]_{hg}& Q\ar[r]\ar[d]& \cdots\ar[r] & P_0\ar[r] \ar@{=}[d] & M\ar@{=}[d]\\ \delta_M:\syz^n M~\ar[r] &P_{n-1}~\ar[r] & \cdots\ar[r] & P_0 \ar[r] & M.}\] Hence, for a given object $X\in\C$, we will have a commutative square
{\footnotesize\[\xymatrix{\Ext^{n+1}(X, M)~\ar[r]^{\cong} \ar@{=}[d]& \Ext^{2n+1}(X, {\syz'}^nM)\ar[d]^{ \h\g}\\ \Ext^{n+1}(X, M)~\ar[r]^{\cong} &\Ext^{2n+1}(X, \syz^nM).}\]}As $hg$ factors through the $n$-projective object $P$, the right column is zero, and then $\Ext^{n+1}(X, M)=0$, meaning that $M\in n$-$\inj\C$. Therefore, $M\in n$-$\proj\C$, because $\C$ is $n$-Frobenius.\\
$(4\Rightarrow 1)$: This implication is clear. So the proof is finished.
\end{proof}

\begin{prop}\label{cp}(1) Let $f:M\rt N$ be a morphism in $\si$. Then for any $X\in\C$, ${\C_{\p}}(N, X)\st{\bar{\f}}\rt{\C_{\p}}(M, X)$ is an isomorphism.\\ (2) Let $M\st{f}\rt N\st{g}\rt K$ be a conflation in $\C$. Then, for any object $X\in\C$, there exists an exact sequence
$${\C_{\p}}(K, X)\st{\bar{\g}}\lrt{\C_{\p}}(N, X)\st{\bar{\f}} \lrt{\C_{\p}}(M, X).$$
\end{prop}
\begin{proof}The first statement is clear. The second one follows from Corollary \ref{qo} and Proposition \ref{pprop}.
\end{proof}

\begin{s}Let $M,N\in\C$ and let $\syz M\rt Q\rt M$ and $\syz N\rt P\rt N$ be two syzygy sequences of $M$ and $N$, respectively. We would like to define an induced map $\syz:\C_{\p}(M, N)\rt\C_{\p}(\syz M, \syz N)$. In this direction, we must define a map $$\syz:(\Ext^n(M, \syz^nN)/{\p}, \delta_N)\lrt(\Ext^n(\syz M, \syz^{n+1}N)/{\p}, \delta_{\syz N}).$$ One should note that, Proposition \ref{ind} allows us to take $\delta_N:=\syz^nN\rt P_{n-1}\rt\cdots\rt P_0\rt N$ and $\delta_{\syz N}:=\syz^{n+1}N\rt P_n\rt P_{n-1}\rt\cdots\rt P_1\rt\syz N.$ Consider the natural isomorphisms $\Ext^n(M, \syz^nN)/{\p}\cong\Ext^{n+1}(M, \syz^{n+1}N)\cong\Ext^n(\syz M, \syz^{n+1}N)/{\p}$, where the first isomorphism holds true, because of Proposition \ref{pprop} and the second one comes from \ref{ccor}. Denoting the composition of these isomorphisms by $\theta$, we define $\syz((\bar{\ga}, \delta_N)):=(\theta(\bar{\ga}), \delta_{\syz N})$. Indeed, we have the following result, which is analogous to the well-known result in the classical stable category of a Frobenius category.
\end{s}
\begin{theorem}\label{syziso}With the notation above, there is an induced map $\syz:\C_{\p}(M, N)\rt\C_{\p}(\syz M, \syz N)$ which is isomorphism.
\end{theorem}

\begin{rem}Let $\L$ be the subcategory of $\coh(\x)$ consisting of all locally free sheaves. As observed in Proposition \ref{locally}, $\C(\L)$ is an $n$-Frobenius subcategory of $\coh(\x)$ with $n$-$\proj\C(\L)=\L$. So considering the phantom stable category $\C(\L)_{\p}$, we have that an object $\F\in\C(\L)$ is locally free if and only if $\C(\L)_{\p}(\F, -)=0=\C(\L)_{\p}(-, \F)$. Next assume that $\x$ is a Grenstein scheme, i.e. all its local rings are Gorenstein local rings. Then, for any $\F\in\coh(\x)$, $\syz^d\F\in\C(\L)$, where $d=\dim\x$, see \cite[Theorem 2.2.3]{ajs}. Thus, we may infer that $\x$ is singular if and only if $\C(\L)_{\p}=0$.
\end{rem}

{\bf{Acknowledgements.}} The authors are grateful to Sergio Estrada and Rasool Hafezi for their reading and suggestions on the first draft of the paper. The authors would like to thank the referee for his/her
careful reading and valuable comments.





\begin{thebibliography}{1}

\bibitem{ajs}
{\sc J. Asadollahi, F. Jahanshahi and Sh. Salarian}, Complete cohomology and Gorensteiness of schemes, {\em J. Algebra} {\bf 319} (2008), 2626--2651.

\bibitem{ab}
{\sc M. Auslander and M. Bridger,} Stable module theory, {\em Mem. Amer. Math. Soc.} {\bf 94} (1969).

\bibitem{as}
{\sc M. Auslander and O. Solberg,} {Relative homology and representation theory I, relative homology and homologically finite categories,} {\em Comm. Alg.} {\bf 21} (9) (1993), 2995--3031.

\bibitem{be}
{\sc D. Benson,} Phantom maps and modular representation theory III, {\em J. Algebra} {\bf 248} (2002) 747-754.

\bibitem{be1}
{\sc D. Benson, Ph. G. Gnacadja,} Phantom maps and modular representation theory I, {\em Fund. Math.} {\bf 161} (1999) 37-91.

\bibitem{be2}
{\sc D. Benson, Ph. G. Gnacadja,} Phantom maps and modular representation theory II,{\em  Algebr. Represent. Theory} {\bf 4} (2001) 395-404.



\bibitem{buh}
{\sc T. Buhler}, Exact categories, {\em Expo. Math.} {\bf 28}(1) (2010), 1-69.

\bibitem{cfh}
{\sc L. W. Christensen, A. Frankild and H. Holm,} On Gorenstein projective, injective and flat dimensions–a
functorial description with applications, {\em J. Algebra} {\bf 302} (2006), 231-279.


\bibitem{et}
{\sc I. Emmanouil and O. Talelli,} Finiteness criteria in Gorenstein homological algebra,{\em  Trans. Amer. Math.
Soc.} {\bf 366} (2014), 6329-6351.

\bibitem{ee}
{\sc E. Enochs, S. Estrada, J. R. García-Rozas and L. Oyonarte}, Flat and cotorsion quasi-coherent sheaves. Applications,
{\em Algebras and representation theory} {\bf 7} (2004), 441-456.

\bibitem{e}
{\sc E. Enochs and J. R. Garcia Rozas}, Flat covers of complexes, {\em J. Algebra} {\bf 210} (1998), 86-102.


\bibitem{fght}
{\sc X. H. Fu, P. A. Guil Asensio, I. Herzog and B. Torrecillas}, {Ideal approximation theory}, {\em Adv. in Math.} {\bf 244} (2013), 750--790.

\bibitem{gn}
{\sc Ph.G. Gnacadja,} Phantom maps in the stable module category, {\em J. Algebra} {\bf 201} (2) (1998) 686-702.

\bibitem{har}
{\sc R. Hartshorne,} Algebraic Geometry, Grad. Texts in Math., vol. 52, Springer, Berlin, Heidelberg, New York, 1977.

\bibitem{ha1}
{\sc D. Happel,} Triangulated categories in the representation theory of finite dimensional algebras, Cambridge University Press,1988.


\bibitem{hln}
{\sc M. Herschend, Y. Liu and  H. Nakaoka} $n$-exangulated categories (I): Definitions and fundamental properties, {\em J. Algebra} {\bf 570} (2021), 531-586.
\bibitem{hln1}
{\sc M. Herschend, Y. Liu and  H. Nakaoka} $n$-exangulated categories (II): Constructions from $n$-cluster tilting
subcategories, {\em J. Algebra} {\bf 594}  (2022), 636-684.

\bibitem{he}
{\sc I. Herzog,} The phantom cover of a module, {\em Adv. Math.} {\bf 215} (1) (2007) 220-249.

\bibitem{hext}
{\sc I. Herzog,} Contravariant functors on the category of finitely presented modules, {\em Israel J. Math.} {\bf 167}  (2008), 347-410.

\bibitem{hs}
{\sc E. Hosseini and Sh. Salarian,} A cotorsion theory in the homotopy category of flat quasi-coherent
sheaves, {\em Proc. Amer. Math. Soc.} {\bf 141} (3) (2013), 753-762.


\bibitem{ja}
{\sc G. Jasso,} $n$-Abelian and $n$-exact categories, {\em  Math. Z.} {\bf 283} (2016), 703-759.

\bibitem{gko}
{\sc C. Geiss, B. Keller and S. Oppermann,} $n$-Angulated categories, {\em J. Reine Angew. Math.} {\bf 675} (2013),
101-120.

\bibitem{ke}
{\sc B. Keller,} Chain complexes and stable categories, {\em Manuscripta Math.} {\bf 67} (4) (1990) 379--417.
\bibitem{ke1}
{\sc  B. Keller,} Derived categories and their uses, in: Handbook of algebra, vol. 1, North-Holland, Amsterdam, 1996, pp. 671–701.

\bibitem{lz}
{\sc Y. Liu and P. Zhou,}  Frobenius $n$-exangulated categories, {\em J. Algebra} {\bf 559} (2020), 161-183.

\bibitem{mao}
{\sc L. Mao,} Higher phantom and $\Ext$-phantom morphisms, {\em J. Algebra Appl.} {\bf 17} (2018), no. 1, 1850012, 15 pp.

\bibitem{mao1}
{\sc L. Mao,} Higher phantom  morphisms with respect to a subfunctor of $\Ext$, {\em Algebr Represent Theor} {\bf 22} (2019), 407-424.

\bibitem{mc}
{\sc C. A. McGibbon,} Phantom maps, in: Handbook of Algebraic Topology, Elsevier Science B.V., Amsterdam, 1995, pp. 1209-1257.

\bibitem{mit}
{\sc B. Mitchel,}  Theory of categories, Pure and applied Math. A series of monographs and textbooks, vol. 17, Elsevier, Academic Press, New York, London, 1965.
\bibitem{mu}
{\sc D. Murfet,}  The  mock  homotopy  category  of  projectives  and Grothendieck  duality,  PhD  thesis, available online at http://www.therisingsea.org/thesis.pdf, 2007.


\bibitem{np}
{\sc H. Nakaoka and Y. Palu,} Extriangulated categories, Hovey twin cotorsion pairs and model structures,
{\em Cah. Topol. G$\acute{e}$om. Diff$\acute{e}$r. Cat$\acute{e}$g.} {\bf 60} (2) (2019), 117-193.


\bibitem{ne}
{\sc A. Neeman,} The Brown representability theorem and phantomless triangulated categories, {\em J. Algebra} {\bf 151} (1) (1992), 118-155.

\bibitem{or}{\sc D. Orlov,} {Triangulated categories of singularities and D-branes in Landau-Ginzburg models}, {\em  Proc. Steklov Inst. Math.} {\bf 246}, no. 3 (2004), 227-248.

\end{thebibliography}
\end{document}